\documentclass{article}   	
\usepackage{geometry}               
\usepackage{graphicx}
\usepackage[usenames,dvipsnames]{xcolor}
\usepackage{amssymb,amsmath,amsthm,amsfonts}
\usepackage{commath} 
\usepackage[integrals]{wasysym}  
\usepackage{bm}                     
\usepackage[english]{babel}
\usepackage[latin1]{inputenc}
\usepackage[colorlinks]{hyperref}
\hypersetup{linkcolor=blue,citecolor=blue,filecolor=black,urlcolor=blue}

\usepackage[sc]{mathpazo}
\usepackage{dutchcal}      
\usepackage{stmaryrd}  
\usepackage{xfrac} 

\newtheorem{proposition}{Proposition}[section]
\newtheorem{theorem}[proposition]{Theorem}
\newtheorem{corollary}[proposition]{Corollary}
\newtheorem{lemma}[proposition]{Lemma}

\theoremstyle{definition}

\theoremstyle{remark}
\newtheorem{remark}[proposition]{Remark}
\numberwithin{equation}{section}

\newcommand{\eps}{\varepsilon}

\newcommand{\R}{{\mathbb{R}}}

\newcommand{\dx}{\, \dif x}
\newcommand{\dxdt}{\dif x \dif t}
\newcommand{\dy}{\, \dif y}
\newcommand{\dyds}{\dif y \dif s}
\newcommand{\dytds}{\dif \tilde y \dif s}

\newcommand{\den}{8}

\newcommand{\dcal}{{\mathcal{d}}}

\newcommand{\Fcal}{{\mathcal{F}}}

\newcommand{\Mcal}{{\mathcal{M}}}

\newcommand{\Kcal}{{\mathcal{K}}}

\newcommand{\ca}{\textsf{(a)}}
\newcommand{\cb}{\textsf{(b)}}

\DeclareMathOperator{\divv}{div}

\DeclareMathOperator{\tr}{tr}

\newcommand{\ssubset}{\subset\joinrel\subset}

\title{On the improvement of H\"older seminorms in superquadratic Hamilton-Jacobi equations}

\author{Marco Cirant}

\begin{document}
 \maketitle
 
\begin{abstract} 
We show in this paper that maximal $L^q$-regularity for time-dependent viscous Hamilton-Jacobi equations with unbounded right-hand side and superquadratic $\gamma$-growth in the gradient holds in the full range $ q > (N+2)\frac{\gamma-1}\gamma$. Our approach is based on new $\frac{\gamma-2}{\gamma-1}$-H\"older estimates, which are consequence of the decay at small scales of suitable nonlinear space and time H\"older quotients. This is obtained by proving suitable oscillation estimates, that also give in turn some Liouville type results for entire solutions.
\end{abstract}
%
\section{Introduction}

The goal of the present paper is to address a problem of maximal $L^q$-regularity for semilinear equations of Hamilton-Jacobi (HJ) type
\begin{equation}\label{hj0}
-\partial_t u - \Delta u + h(x, t) |Du|^\gamma = f(x,t) \qquad \text{on $Q=\Omega \times (0,T)$}
\end{equation}
where $\Omega \subset \R^N$ is a bounded domain, $N\ge 1$, $h$ is a continuous function which does not change sign (and which will be assumed here to be strictly positive), $f$ is given in the Lebesgue space $L^q(Q)$, and the nonlinear gradient term is \textit{superquadratic}, i.e. $\gamma > 2$. We will show that for strong solutions $u$ in the parabolic Sobolev space $W^{2,1}_q$, it is possible to control (locally) the $L^q$-norms of $\partial_t u$, $D^2 u$ and $|Du|^\gamma$ with respect to the $L^q$-norm of $f$.

Equation \eqref{hj0} appears in several contexts, such as stochastic optimal control, where it is naturally backward in time. The regularity result we obtain here sits in between the realms of weak solutions (developed mostly below the so-called natural growth $\gamma = 2$), and viscosity solutions. It is impossible to summarize here the many, many contributions on the study of \eqref{hj0} and its quasi-linear generalizations, but it is rather well known how to deal with weak solutions, which involve the information $|Du|^\gamma \in L^1$, or viscosity ones, that structurally require $f$ to be at least continuous. The problem of maximal regularity in $L^q$, which has a genuine  flavor of \textit{strong solutions}, has been recently considered in \cite{CGpar}. In the superquadratic case $\gamma > 2$, that will be addressed here, previous results are not complete, since maximal regularity is expected to hold up to
\[
q > q_0 = (N+2)\frac{\gamma-1}{\gamma},
\]
while $q \ge (N+2)\frac{\gamma-1}{2}$ only has been covered in \cite{CGpar} (see also \cite{GoffiFrac} for results on nonlocal HJ equations). Here, we deal with the full range $q > q_0$. This extension requires a substantial improvement of the methods developed previously.

There is an elliptic counterpart of maximal regularity for HJ equations, which has been addressed in \cite{CG4, CV, G, GP}. For $q > N(\gamma-1)/\gamma$, that was conjectured to hold by P.-L. Lions a decade ago, and it has been achieved using nonlinear methods (such as the Bernstein method) \cite{CG4, G, GP} or blow-up techniques \cite{CV}. The stationary problem has deep differences with respect to the parabolic one, and ideas developed for the former have by no means obvious adaptations to the latter. For example, blow-up techniques are based on the validity of Liouville theorems. For stationary HJ equations, these are known from the eighties \cite{Lions85, SP}, and they are so powerful that they hold without any information on the growth of solutions at infinity. On the other hand, Liouville theorems for parabolic problems are known under rather restrictive conditions on the growth of the solutions \cite{SZ}, and they are even false for solutions having a linear behavior at infinity (see for example Remark \ref{betterliou}). The problem is even more delicate beyond the natural growth $\gamma > 2$, which is addressed here, where the diffusive part of the equation is typically regarded as negligible (indeed, several ``unnatural'' phenomena which are typical of first-order equations arise, see e.g. \cite{DAP, PSou}). Heuristically, many difficulties come from the very different behavior at small scales of the two operators $-\partial_t u - \Delta u $ and $-\partial_t u -  |Du|^\gamma$, and that is why one usually ``chooses'' between one of them, and regards the rest of the equation as a perturbation. Nevertheless, we wish to show that their effects are somehow intertwined at the $\alpha_0$-H\"older scale, where $\alpha_0 = \frac{\gamma-2}{\gamma-1}$.

Here, our way to maximal regularity goes indeed through the study of H\"older regularity, which is intimately connected with Liouville theorems, and in all of the three aspects our results are new, up to our knowledge. A few works, such as \cite{CSil, SV18}, addressed the problem of H\"older regularity for parabolic problems with $\gamma > 2$ and unbounded $f$, although with a perturbative viewpoint which excluded possible regularization effects coming from the diffusion, and under the sign condition $f \ge 0$ (see Remark \ref{fsign} on further comments on the issues related to the unboundedness of $f$ from below). Let us also mention \cite{CarCan}, where (first-order) counterexamples suggest that without any control on the modulus of continuity of $h$, there might be limitations on the H\"older regularity of solutions. Since our goal is to achieve $\alpha_0$-H\"older (which might be very close to Lipschitz, for large $\gamma$), we will admit estimates to depend on the modulus of continuity of $h$. First-order problems are also considered in \cite{CPT}, where Sobolev regularity of $Du$ is studied with unbounded $f$; it is worth noting that even without diffusion, the HJ equation works at different scales in space and time (that is why \cite{CPT} uses techniques a l\`a DiBenedetto, where cylinders depending on the solution itself are employed). Here, we also need to take care of the effects of the diffusion.

\smallskip

Let us now state the main result. Throughout the paper,
\[
\gamma > 2, 
\qquad h \in C(\overline \Omega \times [0,T]), \qquad 0 < h_0 \le h(x, t) \le h_1.
\]
\begin{theorem}\label{maxreg} Let $q_0 < q < N+2$. For every $K > 0$ and $Q' \ssubset Q$, there exists $C$ depending on $K, Q', Q, h, q_0$ such that if $u \in W^{2,1}_q(Q)$ solves \eqref{hj0} in the strong sense, with $\|u\|_{L^\infty(Q)}, \|f\|_{L^q( Q)} \le K$, then
\[
\|\partial_t u\|_{L^q( Q')} + \|D^2 u\|_{L^q( Q')} + \| |Du|^\gamma \|_{L^q( Q')}\le C.
\]
\end{theorem}

The statement is also true when $q \ge  N+2$, see Remark \ref{beyondNplus2}. As we previously announced, the proof of this result is a consequence of some H\"older regularity properties, which we state here as follows:
\begin{theorem}\label{alpharegintro} Let $q_0 < q < N+2$. For every $K > 0$ and $Q' \ssubset Q$, there exists $C$ depending on $K, Q', Q, h, q_0$ such that if $u \in W^{2,1}_q(Q)$ solves \eqref{hj0} in the strong sense, with $\|u\|_{L^\infty(Q)}, \|f\|_{L^q(Q)} \le K$, then
\[
\frac{|u(x,t)-u(x',t')|}{|x-x'|^\alpha + |t-t'|^{\sfrac\alpha 2}} \le C
\]
for all $(x,t),(x',t') \in Q'$, where $\alpha = 2- \frac{N+2}{q}$.

If $q = q_0$, then the same conclusion holds, but $K$ might depend on $f$ not only through its norm (in fact, $K$ is independent of $f$ whenever $f$ varies in a uniformly integrable subset of $L^{q}(Q)$).
\end{theorem}

This theorem will be proven in two steps. First, for $q = q_0$ in Theorem \ref{alpha0reg}. Then, for $q > q_0$ in Theorem \ref{alphareg}. While the general strategy will be presented in detail in the second part of the introduction, we already mention that the first step is the core of the paper. When $q = q_0$, then $\alpha = \alpha_0= \frac{\gamma-2}{\gamma-1}$. We may then regard Theorem \ref{alpharegintro} as a parabolic version of the $\alpha_0$-H\"older regularity obtained in \cite{DAP} for stationary problems (actually for general divergence-form equations, for subsolutions, and up to the boundary...). The intermediate passage between the two steps $q=q_0$ and $q > q_0$ will be the Liouville Theorem \ref{liouville}.

The techniques used here have their roots in \cite{CGpar} (see also \cite{CG2, Gomesbook}), in \cite{CV} and in \cite{CSil}. However, new and crucial ingredients are introduced, such as the use of ``nonlinear'' H\"older seminorms, and the study of their decay (or improvement) in smaller and smaller cylinders. All the estimates obtained here are \textit{local} in nature, but they could be extended up to the parabolic boundary of $Q$ under suitable boundary conditions; we will not pursue their derivation here. Moreover, more general nonlinear functions of $Du$ are admissible, see for instance Remark \ref{remass}. On the other hand, the presence of $(x,t)$-dependent diffusions cause additional difficulties, and they will not be treated here.

\smallskip

Let us mention that we do not know whether maximal regularity holds below the exponent $q_0$, as no counterexamples like in the stationary case are at this stage available. Nevertheless, we strongly believe that it is not possible to have maximal regularity below $q_0$, and $L^\infty$ estimates might even fail ($f$ need not be bounded below here). Besides the fact that $q_0$ is the parabolic analogue of the sharp stationary exponent, $q_0$ is ``critical'' (with respect to scaling properties of \eqref{hj0}), and H\"older estimates ``see'' properties of $f$ beyond its norm... 
Finally, we do not know at the moment whether it is possible or not to obtain \textit{universal} estimates, as in parabolic problems with nonlinear zero-th order term \cite{PQS}, or for stationary HJ equations \cite{CV}; our H\"older and maximal regularity bounds depend indeed the $L^\infty$-norm of $u$. 

\smallskip

Before entering into the details of the proofs, let us say a few words on the structure of the paper. Section \ref{sec:osc} will be devoted to the proof of crucial oscillation estimates for HJ equations. In Section \ref{sec:hol}, these will be used to obtain the first step of $\alpha_0$-H\"older regularity, while in Section \ref{sec:lio} we will show how to derive a Liouville type result. In Section \ref{sec:hol2} the full range of H\"older regularity will be obtained (that is, for any $\alpha > \alpha_0$), and in Section \ref{sec:max} we will finally reach the maximal regularity result. The Appendix \ref{onthefp} contains several useful results on linear equations which will be used thoroughly.

\medskip
\textbf{Strategy of the proofs.} The first main idea to achieve $W^{2,1}_q$ estimates can be described as follows; standard parabolic embeddings read
\[
W^{2,1}_q \hookrightarrow W^{1,0}_p \hookrightarrow C^{\alpha, \sfrac\alpha2},
\]
where $\frac1p = \frac1q - \frac1{N+2}$ and $\alpha=2-\frac{N+2}q$ (for $q > \frac{N+2}2$). In fact, since we consider solutions in $W^{2,1}_q$, we look for a priori estimates (at all scales, from H\"older to $W^{2,1}_q$). If $q > q_0$, then a priori bounds can be obtained by following the arrows in the ``opposite way''; if one is able indeed to control $|Du|^\gamma$ in $L^{q}$, then standard parabolic regularity yields $W^{2,1}_q$ estimates. Furthermore, if one has $\alpha$-H\"older bounds, then an additional interpolation argument yields back $W^{2,1}_q$ estimates. This step is carried out by Proposition \ref{prop:holdertoW2q}. We stress that this is possible only if $q > q_0$. In other words, the ground $\alpha$-H\"older control allows to treat the HJ equation in a perturbative manner, and the effect of the nonlinear Hamiltonian term $|Du|^\gamma$ can be absorbed by the usual gain of regularity coming from the linear part of the equation.

\smallskip

The second main ingredient is the improvement of the $\alpha_0$-H\"older bounds to $\alpha$-H\"older bounds, which is done by Theorem \ref{alphareg}. Here, we follow an idea developed in \cite{CV} for stationary problems: if there are sequences of solutions whose $\alpha$-H\"older seminorm explodes, then a suitable rescaling argument allows to construct a nonconstant ``entire'' solution (on $\R^N \times (0,\infty)$), which is in contradiction with a Liovuille Theorem, forcing in fact such a solution to be constant. There are a few points that are worth noticing: first, the rescaling here preserves the linear structure ( $(x,t) \mapsto (rx, r^2t)$ ), but the Hamiltonian might not vanish in the limit problem. For this reason we need a Liouville theorem for homogeneous HJ equations. Such a Liouville theorem cannot hold if the solution has linear growth in the $x$ variable (see Remark \ref{betterliou}), that is why we need to start with $\alpha_0$-H\"older bounds, that give in the limit solutions with sublinear growth.

\smallskip

The Liouville property (Theorem \ref{liouville}) and the $\alpha_0$-H\"older bounds (Theorem \ref{alpha0reg}) are the main novelties in this paper, and they are deduced as follows. It is classical that H\"older regularity can be obtained whenever the oscillation of $u$ in a cylinder is smaller (by a fixed factor) than the oscillation of $u$ itself in a larger cylinder (at all scales). This principle, which has been used in the context of HJ equations in \cite{CSil}, typically leads to H\"older regularity with rather implicit exponent, which might be actually very close to zero. Here, we need the precise exponent $\alpha_0$ (which approaches $1$ as $\gamma \to \infty$). Therefore, we focus on the following principle: 
\begin{gather*}
\textit{the $\alpha_0$-H\"older seminorm in a cylinder is strictly smaller than} \\
\textit{the $\alpha_0$-H\"older seminorm in a larger cylinder,}
\end{gather*}
which can be seen as an ``improvement of $\alpha_0$-H\"older" regularity. 

Practically, we will rescale sequences of solutions having (by contradiction) growing $\alpha_0$-H\"older seminorm in a way that it is exactly $1$ in a cylinder $Q_0$ of radius one, while the global $\alpha_0$-H\"older seminorm remains bounded by a constant. We then prove that the control on the $\alpha_0$-H\"older seminorm at the boundary of a \textit{suitable} large cylinder $Q_1$ implies that the oscillation of $u$ on $Q_0$ needs to be strictly smaller than $1$, which is impossible. To perform this fundamental step, we will need to introduce a nonstandard ``nonlinear'' $\alpha_0$-H\"older seminorm, which will be denoted below by $\llbracket \cdot \rrbracket_{\alpha_0}$; indeed, if we rescale in a way that the full operator
\[
-\partial_t u - \Delta u + |Du|^\gamma,
\]
is preserved, then the standard $\alpha_0$-H\"older parabolic seminorm remains invariant. We will then rescale in a way that the transport part $-\partial_t u + |Du|^\gamma$ is kept unchanged, while the space-time H\"older quotients
\begin{equation}\label{stranaholder}
\llbracket u \rrbracket_{\alpha_0} \approx \max\left\{ \, \frac{|u(x,t) - u(\bar x,t)|}{|x-\bar x|^\alpha}, \, \left( \frac{|u(x,t) - u(x,\bar t)|}{|t-\bar t|^{\sfrac{\alpha}2}}\right)^{\sfrac2\gamma} \right\}
\end{equation}
are renormalized. Note that in this way the rescaled equation has vanishing viscosity. The oscillation core estimate which is then used to conclude is contained in Proposition \ref{mainstima}. Since the proof of the proposition is rather technical, we try to sketch here the main ideas involved (and briefly describe how to handle the vanishing viscosity, which deteriorates second-order regularization effects).

As $u$, and its scaling $w$, solve an HJ equation, we have at our disposal a representation formula which arises in stochastic optimal control theory: if for simplicity $h$ is constant, then
\[
w(x,0) = \inf_{b_s} \, \mathbb E \int_0^\tau \ell |b_s|^{\gamma'} + f(X_s, s) ds + \mathbb E w(X_\tau, \tau),
\]
where the infimum is taken among stochastic trajectories $d X_s = b_s ds + \sqrt{2 \sigma} dB_s$ originating from $(x,0)$ and $\tau$ is the first exit time from the cylinder $Q = B_R \times (0,T)$. In PDE terms, the previous formula reads
\begin{equation}\label{valuefunction}
w(x,0) =  \iint_Q \ell |b|^{\gamma'} m + f m \dyds +  \int_{B_R} w(T) m(T) \dy - \iint_{\partial B_R  \times (0,T) } w Dm \cdot \nu \dyds,
\end{equation}
where $b = - h \gamma |Dw|^{\gamma-2}Dw$, $\nu$ is the outward normal vector to $\partial B_R$, and
\begin{equation}\label{fokkeri}
\begin{cases}
\partial_s m - \sigma \Delta m - \divv(b m) = 0 & \text{on $Q$}, \\
m|_{\partial B_R \times (0,\tau)} = 0, \ \  m(0) = \delta_x.
\end{cases}
\end{equation}
Formula \eqref{valuefunction} has been used extensively to deduce regularity properties of $w$. Besides its interpretation in optimal control, it can be obtained using a sort of duality between \eqref{hj0} and \eqref{fokkeri}; such a duality is at the core of the so-called (nonlinear) adjoint method, which allows to derive regularity of $w$ by regularity properties of $m$. This strategy has been proposed in \cite{Evans} for HJ equations. In particular, H\"older seminorms of $u$ are related to variations with respect to $x$ in \eqref{valuefunction}, that necessarily involve variations of $m$ with respect to $x$, by the presence of $\iint f m$. Unfortunately, the only available information on $m$ is a bound in $L^{q_0'}$ (see Lemma \ref{lemmamq0}); bounds on the derivatives of $m$ are possible in smaller spaces only, and that is why H\"older and maximal regularity are not achieved up to $q = q_0$ in \cite{CGpar}. For the reader who is experienced with the adjoint method for HJ equations, the main challenge here is to ``differentiate'' $u$ in \eqref{valuefunction} with respect to $x$ without ``differentiating'' $m$ in $\iint f m$ ! This is a substantial step which is developed in the present paper. 

\smallskip
To compare $w(x,0)$ and $w(x+h,0)$, we just ``bend'' $m$ to a new $\tilde m$, which starts from $\delta_{x+h}$ and coincides with $m$ at final time. Such new $\tilde m$ is used in the representation formula for $w(x+h,0)$ (since the new $\tilde m$ is ``suboptimal'', the equality becomes an inequality in \eqref{valuefunction}), which is then compared with \eqref{valuefunction} for $w(x,0)$. Note that $w(x+h,0) - w(x,0)$ is renormalized to be $1$, and there is no need to divide by $h$. The discrepancies in the two formulas mainly involve the difference between the two Lagrangian terms $\iint |b|^{\gamma'}m$, the so-called running costs $\iint f m$, and the boundary contributions $ \iint_{\partial B_R  \times (0,T) } w Dm \cdot \nu$. These depend on the dimensions $R,T$ of the large cylinder; loosely speaking, Lagrangian terms can be made small by choosing $T$ large, and $\iint f m$ vanishes since the $L^{q_0}$-norm of $f$ can be made small at small scales. The boundary terms are definitely the most delicate ones; the key point is to realize that $ \iint_{\partial B_R  \times (0,T) } Dm \cdot \nu$, which represent the total density leaving the cylinder from its lateral boundary, is small if $R$ is large, and it can compensate $w$, which might be large (if the large cylinder is chosen carefully enough). Similar arguments allow to control the oscillation of $w$ in time. 

\smallskip
A few comments are now in order. As we said, the scaled operator has vanishing viscosity
\[
-\partial_t u - \sigma \Delta  u + |Du|^\gamma, \qquad \sigma = o(1),
\]
which allows to exploit the first order regularization effects, but of course brings a deterioration of linear regularity. Nevertheless, by the particular way in which space-time is scaled we can still use the (small) linear contribution: the deterioration of the $L^{q_0'}$-norm of $m$ is compensated by rapid vanishing of the $L^{q_0}$-norm of $f$, so that the whole term can $\iint f m$ be controlled in the scaling process. To carry out the whole procedure, we stress again that we need to work with ad-hoc H\"older quotients, and that both the linear and nonlinear part of the equation play a major role in regularization; the argument is by no means perturbative. Finally, the reader will see that we need an additional parameter $z$ in the ``nonlinear'' H\"older seminorm to fine tune the argument. This is mainly because we are able to shrink the \textit{time} H\"older seminorm from $z$ to $z/2$ only if $z$ is large, while $z$ plays basically no role when shrinking seminorms in \textit{space} (provided that cylinders are fat enough). Since our goal is to obtain local estimates, suitable distance functions will appear also in the computations, but this will be just an annoying technical point.

\medskip
\textbf{Notations.}

\begin{itemize}
\item $q_0 = \frac{N+2}{\gamma'}$, $\alpha_0 = \frac{\gamma-2}{\gamma-1}$. Throughout the paper, $\gamma > 2$, $q_0 \le q < N+2$ and $\alpha = 2-\frac{N+2}q$, unless otherwise specified. Note that $q$ is always larger than $1+\frac N 2$.
\item $|\cdot|$ will denote the usual Euclidean distance in $\R^N$; $B_r(x) = \{y \in \R^N : |x-y| < r\}$. $d(x, \Omega)$ be the Euclidean distance from $x \in \R^N$ to the set $\Omega \subset \R^N$. For a cylinder $Q = \Omega \times (a,b)$, we will use two different parabolic distances from the (backward parabolic) boundary $ \partial^+ Q = \partial \Omega \times (a,b) \cup \Omega \times \{b\}$:
\[
d{\big((x,t), \partial^+ Q \big)} = d(x, \partial \Omega) + |b-t|^{\sfrac12}, \qquad \dcal_\alpha{\big((x,t), \partial^+ Q \big)} = d^\alpha(x, \partial \Omega) + |b-t|^{\sfrac\alpha\gamma}
\]
($\dcal_\alpha$ will be often used without the subscript $\alpha \in (0,1)$ for brevity).
\item Parabolic cylinders: 
$Q_{r,t} = B_r(0) \times (0,t)$, $Q_{r} = B_r(0) \times (0,r^2)$.
\item For $Q = \Omega \times (a, b)$, $0 < \alpha < 1$, $c > 0$, we define the usual parabolic H\"older seminorm
\[
[u]_{\alpha; Q} = \sup_{(x,t), (\bar x, \bar t) \in Q} \frac{|u(x,t) - u(\bar x, \bar t)|}{(|x-\bar x| + |t-\bar t|^{\sfrac12})^{\alpha}} \, ,
\]
and its wighted version
\[
[u]^{c}_{\alpha; Q} = \sup_{(x,t), (\bar x, \bar t) \in Q} [\min\{ d((x,t), \partial^+ Q ), d((\bar x, \bar t), \partial^+ Q ) \} ]^{c} \frac{|u(x,t) - u(\bar x, \bar t)|}{(|x-\bar x| + |t-\bar t|^{\sfrac12})^{\alpha}}. 
\]
Finally, for any $z > 0$,
\begin{align*}
\llbracket u\rrbracket^x_{\alpha; Q} = & \sup_{(x,t), (\bar x,t) \in Q} \min\{ \dcal_\alpha{((x,t), \partial^+ Q )}, \dcal_\alpha{((\bar x,t), \partial^+ Q )} \} \frac{|u(x,t) - u(\bar x,t)|}{|x-\bar x|^\alpha}, \\
\llbracket u\rrbracket^t_{\sfrac{\alpha}2; Q} = & \sup_{(x,t), (x,\bar t) \in Q} [\min\{ \dcal_\alpha{((x,t), \partial^+ Q )}, \dcal_\alpha{(( x, \bar t), \partial^+ Q )} \} ]^{\sfrac{\gamma}2} \frac{|u(x,t) - u(x,\bar t)|}{|t-\bar t|^{\sfrac{\alpha}2}}, \\
\llbracket u \rrbracket_{\alpha, z; Q} = & \max\left\{ \llbracket u\rrbracket^x_{\alpha; Q} \, ,  \, (z^{-1}\llbracket u\rrbracket^t_{\sfrac{\alpha}2; Q})^{\sfrac2{\gamma} }\right \}.
\end{align*}
\item $\|\cdot\|_{L^q(X)}$ be the usual norm on the Lebesgue space $L^q(X)$. Parabolic Sobolev spaces:
\[
W^{2,1}_q(X) = \{ \text{measurable $u$} : \partial^r_t D^\beta_x u \in L^q(X) \text{ for any } |\beta| + 2r \le 2\}.
\]
\item Recall that a set $\Fcal \subset L^q(X)$ is $L^q$-uniformly integrable if for any $\eps > 0$, there exists $\delta > 0$ such that
\[
\|f\|_{L^q(E\cap X)} < \eps \qquad \text{whenever $f \in \Fcal$ and $\int_{E\cap X} 1 < \delta$.}
\]
\item $C, C_1, ...$ will denote positive constants, whose value may vary from line to line. Unless otherwise specified, they may depend on $N, \alpha, \gamma, q, Q$ and $h$ (in particular on $h_0, h_1$ and the modulus of continuity of $h$), but \textit{not} on $u$ and $f$.
\end{itemize}

\medskip
\textbf{Acknowledgements.} The author was partially supported by the King Abdullah University of Science and Technology (KAUST) project CRG2021-4674 ``Mean-Field Games: models, theory, and computational aspects'', by the INDAM-GNAMPA grant ``Propriet\`a quantitative e qualitative per EDP non lineari con termini di gradiente'' (2022).

\section{Oscillation estimates}\label{sec:osc}

This section is devoted to the proof of the oscillation estimates, which will be used to show that certain H\"older seminorms decay in smaller and smaller cylinders. These will be also crucial to deduce the Liouville theorem for entire solutions to HJ equations. 

Throughout this section, $w \in W^{2,1}_q(Q_{R+1,\tau})$ (for some $q \ge q_0$) solves on $Q_{R+1,\tau}$ the differential inequalities
\begin{align}
-\partial_s w - \sigma \Delta w + h_0 |Dw|^\gamma \le g, \label{hja} \\
-\partial_s w -  \sigma \Delta w + h_1 |Dw|^\gamma \ge g \label{hjb}
\end{align}
 in the strong sense, for a diffusion factor $0 \le \sigma \le 1$, and $0 < \underline h \le h_0, h_1 \le \bar h$. The function $g$ is given in $L^{q_0}(Q_{R+1,\tau})$.

\begin{proposition}\label{mainstima} Assume that $R^2 \ge \sigma \tau$, and that for some $0<\alpha <1$, $z \ge 1$,
\[
\sup_{(y,s), (y',s) \in \overline{Q}_{R, \tau}} \frac{|w(y,s) - w(y',s)|}{|y-y'|^\alpha} \le 3, \qquad 
\sup_{(y,s), (y, s') \in \overline{Q}_{R, \tau}} \frac{|w(y,s) - w(y, s')|}{|s-s'|^{\alpha/2}} \le 3^{\frac\gamma2} z.
\]
Then there exist $f^0, c_1, C_2, C_3$ depending on $\underline h, \bar h, \alpha, q_0$, and $\Kcal \ge 0$, such that if
\begin{equation}\label{condstima}
\sigma^{-\gamma'\frac{N+1}{N+2}} \|g\|_{L^{q_0}(Q_{R,\tau})}  \le f^0, \qquad z \frac{R^\alpha + \tau^{\frac\alpha2}}{R} \le c_1,
\end{equation}
then for any $|y_0| \le 1$,
\begin{align}
& |w(0,0)-w(0,\tau)| +  \Kcal  \le C_2\left( \tau^{\frac\alpha2} + \tau^{\frac{\alpha_0}2}  +  \tau^{\frac{\alpha}{\gamma-\alpha(\gamma-1)}} +  \tau \frac{R^\alpha + \tau^{\frac\alpha2}}{R}z  \right), \label{test0} 
 \\ & w(y_0,0) - w(0,0)  \le C_3 \Big( \frac{  \Kcal^{1/\gamma}} {\tau^{1/\gamma}} + \frac1{\tau^{\gamma'-1}}  + \nonumber \\
 & \qquad \qquad   +  \sigma^{-\gamma'\frac{N+1}{N+2}} ( \Kcal +\tau^{\alpha_0/2} )  \|g\|_{L^{q_0}(Q_{R+1,\tau})} +  \frac{\tau^{1/\gamma} \Kcal^{1/\gamma'} }R +  \frac{\tau}{R^2} + (\ell_0-\ell_1)\Kcal \Big),\label{xest0}
\end{align}
where $\ell_i = \frac{h_i(\gamma-1)}{(h_i\gamma)^{\gamma'}}$, $i=1,2$
\end{proposition}

The proof will be obtained in several steps. In what follows, will always assume that $R^2 \ge \sigma \tau$.

\begin{lemma}\label{w0wtau}  There exists $C_0$ depending on $\alpha, q_0$  such that
\[
w(0,0)-w(0,\tau) \le C_0\left(\tau^{\frac\alpha2} + \tau \frac{R^\alpha + \tau^{\frac{\alpha}2}}{R^2}z + \sigma^{-\gamma'\frac{N+1}{N+2}} \|g\|_{L^{q_0}(Q_{R,\tau})} \tau^{{\alpha_0}/2} \right).
\]
\end{lemma}

\begin{proof} Let $\mu$ be the solution of
\[
\begin{cases}
\partial_s \mu - \sigma \Delta \mu = 0& \text{on $B_R \times (0,\tau)$} \\
\mu(y,s) = 0 & \text{on $\partial B_R \times (0,\tau)$} \\
\mu(0) = \delta_0
\end{cases}
\]
(the existence of such $\mu$ and its regularity are detailed in Appendix \ref{onthefp}). Testing \eqref{hja} by $\mu$ and the equation for $\mu$ by $w$, and integrating by parts on $Q_{R,\tau}$ yields
\begin{multline*}
w(0,0) + h_0 \iint_{B_R  \times (0,\tau) } |Dw|^\gamma \mu \dyds \le \\ + \iint_{B_R  \times (0,\tau) } g \mu \dyds + \int_{B_R} w(y,\tau) \mu(y,\tau) \dy- \sigma \iint_{\partial B_R  \times (0,\tau) } w D\mu \cdot \nu \dyds.
\end{multline*}
Since $\int_{B_R} \mu(y,\tau) - \sigma \iint_{ \partial B_R \times (0,s) } D\mu \cdot \nu = 1$,
\begin{multline*}
w(0,0) - w(0,\tau) \le \int_{B_R} [w(y,\tau)-w(0,\tau)] \mu(y,\tau) \dy \\
+ \iint_{B_R  \times (0,\tau) } g \mu \dyds
- \sigma \iint_{\partial B_R  \times (0,\tau) } [w-w(0,\tau)] D\mu \cdot \nu \dyds.
\end{multline*}
First, by the growth assumptions on $w$ and Lemma \ref{bdglemma}
\[
\int_{B_R} [w(y,\tau)-w(0,\tau)] \mu(y,\tau) \dy \le 3 \int_{B_R} |y|^\alpha \mu(y,\tau) \le 3 C \left(\sigma \tau \right)^{\alpha/2}.
\]
Secondly, applying Lemma \ref{lemmamq0} 
\[
\iint_{B_R  \times (0,\tau) } g \mu \dyds \le \|g\|_{L^{q_0}(Q_{R,\tau})} \|\mu\|_{L^{q_0'}(Q_{R,\tau})}
\le C \sigma^{-\gamma'\frac{N+1}{N+2}} \|g\|_{L^{q_0}(Q_{R,\tau})}  \sigma^{\gamma'/2}\tau^{\alpha_0/2}.
\]
Finally, again 
by the growth assumptions on $w$ and \eqref{boundaryloss}
\begin{multline*}
\sigma \iint_{\partial B_R  \times (0,\tau) } [w(y,s)-w(0,\tau)] (-D\mu(y,s) \cdot \nu(y)) \dyds \le \\
\sigma 3^{\gamma/2}z (R^\alpha + \tau^{\frac\alpha2}) \iint_{\partial B_R  \times (0,\tau) }  -D\mu(y,s) \cdot \nu(y) \dyds \le 
C z (R^\alpha + \tau^{\frac\alpha2})  \frac{\tau}{R^2},
\end{multline*}
which, together with the previous inequalities, yields the assertion.
\end{proof}

Let us now define
\[
b(y,s) = h_1\gamma|Dw(y,s)|^{\gamma-2} Dw(y,s),
\]
and $m$ be the solution of the dual problem
\[
\begin{cases}
\partial_s m - \sigma \Delta m - \divv(b m) = 0& \text{on $B_R \times (0,\tau)$} \\
m(y,s) = 0 & \text{on $\partial B_R \times (0,\tau)$} \\
m(0) = \delta_0.
\end{cases}
\]

\begin{lemma} There exists $c_0, c_1, C_1, f^0 > 0$ depending on $\bar h, \alpha, q_0$ such that if
\begin{equation}\label{cond0}
\sigma^{-\gamma'\frac{N+1}{N+2}} \|g\|_{L^{q_0}(Q_{R,\tau})}  \le f^0, \qquad z \frac{R^\alpha + \tau^{\frac\alpha2}}{R} \le c_1,
\end{equation}
then
\[
w(0,0)-w(0,\tau) \ge c_0 \iint_{B_R  \times (0,\tau) } |b|^{\gamma'} m \dyds + \Phi,
\]
where
\[
\Phi \ge - C_1\left(\tau^{\frac\alpha2} + \tau^{\frac{\alpha_0}2}  +  \tau^{\frac{\alpha}{\gamma-\alpha(\gamma-1)}} + \tau \frac{R^\alpha + \tau^{\frac\alpha2}}{R}z  \right).
\]
\end{lemma}

\begin{proof} 
Testing \eqref{hjb} by $m$ and the equation for $m$ by $w$, and integrating by parts on $Q_{R,\tau}$ yields
\begin{multline}\label{wn}
w(0,0) \ge h_1(\gamma-1) \iint_{B_R  \times (0,\tau) } |Dw|^\gamma m \dyds + \\+ \iint_{B_R  \times (0,\tau) } g m \dyds+  \int_{B_R} w(y,\tau) m(y,\tau) \dy- \sigma \iint_{\partial B_R  \times (0,\tau) } w Dm \cdot \nu \dyds.
\end{multline}
whence
\[
w(0,0) - w(0,\tau) \ge c_0 \iint_{B_R  \times (0,\tau) } |b|^{\gamma'} m \dyds + \Phi,
\]
where $2c_0 = \frac{\bar h(\gamma-1)}{(\bar h\gamma)^{\gamma'}} \le \frac{h_1(\gamma-1)}{(h_1\gamma)^{\gamma'}} =: \ell_1$, and
\begin{multline*}
\Phi = c_0 \iint_{B_R  \times (0,\tau) } |b|^{\gamma'} m \dyds + 
\int_{B_R} [w(y,\tau)-w(0,\tau)] m(y,\tau) \dy \\
+ \iint_{B_R  \times (0,\tau) } g m \dyds
- \sigma \iint_{\partial B_R  \times (0,\tau) } [w-w(0,\tau)] Dm \cdot \nu \dyds.
\end{multline*}
First, by the growth assumptions on $w$, inequality \eqref{bdglemma2}, and Young's inequality (notice that $\frac\alpha\gamma\big(\frac{\gamma'}\alpha\big)' = \frac{\alpha}{\gamma-\alpha(\gamma-1)} $),
\begin{multline*}
\int_{B_R} [w(y,\tau)-w(0,\tau)] m(y,\tau) \dy \ge - 3 \int_{B_R} |y|^\alpha m(y,\tau) \\ \ge - C \tau^{\alpha/\gamma}   \left( \iint_{B_R  \times (0,\tau) } |b|^{\gamma'} m \dyds  \right)^{\alpha/\gamma'} - C \left(\sigma \tau \right)^{\alpha/2} \\
\ge - \frac{c_0}3 \iint_{B_R  \times (0,\tau) }|b|^{\gamma'} m \dyds - C( \tau^{\alpha/2} + \tau^{\frac{\alpha}{\gamma-\alpha(\gamma-1)} } ) .
\end{multline*}
Secondly, applying Lemma \ref{lemmamq0} and assuming $f^0$ small enough (depending on $c_0$ and $C$ below),
\begin{multline*}
\iint_{B_R  \times (0,\tau) } g m \dyds \ge - \|g\|_{L^{q_0}(Q_{R,\tau})} \|m\|_{L^{q_0'}(Q_{R,\tau})} 
\ge   - f^0  \sigma^{\gamma'\frac{N+1}{N+2}}   \|m\|_{L^{q_0'}(Q_{R,\tau})} \\
\ge -Cf^0 \left( \iint_{ B_R  \times (0,\tau) } |b|^{\gamma'} m \dxdt + \sigma^{\gamma'/2}\tau^{\alpha_0/2} \right) \\
\ge - \frac{c_0}3 \iint_{B_R  \times (0,\tau) }|b|^{\gamma'} m \dyds -  C\sigma^{\gamma'/2}\tau^{\alpha_0/2}.
\end{multline*}
Finally, by the growth assumptions on $w$, \eqref{boundaryloss} and Young's inequality,
\begin{multline*}
\sigma \iint_{\partial B_R  \times (0,\tau) } [w(y,s)-w(0,\tau)] (-Dm(y,s) \cdot \nu(y)) \dyds \ge \\
\ge -\sigma 3^{\gamma/2}z (R^\alpha + \tau^{\frac\alpha2}) \iint_{\partial B_R  \times (0,\tau) }  -Dm(y,s) \cdot \nu(y) \dyds \\
\ge 
- C z (R^\alpha + \tau^{\frac\alpha2})\left( \frac{\tau^{1/\gamma}}R \left( \iint_{B_R  \times (0,\tau) } |b|^{\gamma'} m \dxdt \right )^{1/\gamma'}+  \frac{\tau}{R^2} \right) \\ 
\ge - \overline C z (R^\alpha + \tau^{\frac\alpha2})\left( \frac1R \iint_{B_R  \times (0,\tau) } |b|^{\gamma'} m \dxdt + \frac{\tau}R+  \frac{\tau}{R^2} \right) \\
\stackrel{\eqref{cond0}}{\ge} - \frac{c_0}3 \iint_{B_R  \times (0,\tau) }|b|^{\gamma'} m \dyds - \overline C z (R^\alpha + \tau^{\frac\alpha2})\left(  \frac{\tau}R+  \frac{\tau}{R^2} \right),
\end{multline*}
where $c_1$ in \eqref{cond0} is chosen to be $c_1=c_0/(3\overline C)$. The three estimates above yield the statement.
\end{proof}

\begin{corollary} Under the assumption \eqref{cond0}, for some $c_0, C_2$ depending on $\bar h, \alpha, q_0$,
\begin{equation}\label{testimate}
|w(0,0)-w(0,\tau)| +  \iint_{B_R  \times (0,\tau) } |b|^{\gamma'} m \dyds \le C_2\left(\tau^{\frac\alpha2} + \tau^{\frac{\alpha_0}2}  +  \tau^{\frac{\alpha}{\gamma-\alpha(\gamma-1)}}  +  \tau \frac{R^\alpha + \tau^{\frac\alpha2}}{R}z  \right).
\end{equation}
\end{corollary}

\begin{proof} The previous lemma states that
\[
 c_0 \iint_{B_R  \times (0,\tau) } |b|^{\gamma'} m \dyds \le w(0,0)-w(0,\tau) - \Phi,
\]
hence by Lemma \ref{w0wtau} one concludes (using that $R \ge 1$ and $\sigma^{-\gamma'\frac{N+1}{N+2}} \|g\|_{L^{q_0}(Q_{R,\tau})}  \le f^0$) .
\end{proof}

The previous corollary allows to control the oscillation of $w$ in time. We now take care of the oscillation of $w$ in space. Let $y_0$ be any unit vector, $R, \tau \ge 1$, $R^2 \ge \sigma \tau$ as before, and
\[
\xi_s = \frac{\tau-s}\tau y_0.
\]

\begin{lemma} The following inequality holds.
\begin{multline}\label{wnabove}
w(y_0,0) \le \ell_0 \iint_{B_R  \times (0,\tau) } |b - \xi'_s|^{\gamma'} m \dyds + \iint_{B_R  \times (0,\tau) } g(y + \xi_s, s) m(y,s) \dyds + \\ +  \int_{B_R} w(y,\tau) m(y,\tau) \dy- \sigma \iint_{\partial B_R  \times (0,\tau) } w(y + \xi_s,s) Dm(y,s) \cdot \nu \dyds.
\end{multline}
\end{lemma}

\begin{proof} Since
\[
h_0 |p|^\gamma = \sup_{q \in \R^N} \left\{ p \cdot q - \ell_0 |q|^{\gamma'} \right\}, \qquad \text{where $\ell_0 = \frac{h_0(\gamma-1)}{(h_0\gamma)^{\gamma'}}$},
\]
for any vector field $\tilde b$ we have
\begin{equation}\label{subhj}
-\partial_s w - \sigma  \Delta w + Dw \cdot \tilde b - \ell_0 |\tilde b|^{\gamma'} \le g \qquad \text{on $Q_{R+1,\tau}$}.
\end{equation}
Let us now choose
\[
\tilde b(\tilde y, s) = b(\tilde y - \xi_s, s) - \xi'_s \, = h_1\gamma|Dw(\tilde y - \xi_s,s)|^{\gamma-2} Dw(\tilde y - \xi_s,s) + \frac{y_0}\tau \quad \text{on $B_R(\xi_s) \times (0,\tau)$}.
\]
Test now the inequality \eqref{subhj} by $m(\tilde y - \xi_s, s)$, and integrate by parts on $B_R(\xi_s) \times (0,\tau) \ni (\tilde y, s)$ as follows:
\begin{multline*}
-\iint_{B_R(\xi_s)  \times (0,\tau) } \partial_s w(\tilde y, s) m(\tilde y - \xi_s, s) \dytds = 
-\iint_{B_R  \times (0,\tau) } \partial_s w(y + \xi_s, s) m(y, s) \dyds \\
= -\iint_{B_R  \times (0,\tau) } \partial_s (w(y + \xi_s, s)) m(y, s) \dyds + \iint_{B_R  \times (0,\tau) } Dw(y + \xi_s, s) \cdot \xi_s' m(y, s) \dyds \\
= \iint_{B_R  \times (0,\tau) } w(y + \xi_s, s) \partial_s m(y, s) \dyds + \iint_{B_R  \times (0,\tau) } Dw(y + \xi_s, s) \cdot \xi_s' m(y, s) \dyds \\ 
- \int_{B_R} w(y + \xi_\tau, \tau) m(y, \tau) \dy  + w(\xi_0, 0).
\end{multline*}

Secondly,
\begin{multline*}
-\iint_{B_R(\xi_s)  \times (0,\tau) } \sigma \Delta w(\tilde y, s) m(\tilde y - \xi_s, s) \dytds =
\\ = -\iint_{B_R  \times (0,\tau) } \sigma \Delta w(y + \xi_s, s) m(y, s) \dyds \\
 = -\iint_{B_R  \times (0,\tau) } \sigma \Delta  m(y, s) w(y + \xi_s, s) \dyds  + \sigma \iint_{\partial B_R  \times (0,\tau) } w(y + \xi_s,s) Dm  \cdot \nu \dyds.
\end{multline*}

Regarding the third term,
\begin{multline*}
\iint_{B_R(\xi_s)  \times (0,\tau) } Dw(\tilde y,s) \cdot \tilde b(\tilde y,s) m(\tilde y - \xi_s, s) \dytds = \\
= \iint_{B_R  \times (0,\tau) } Dw(y + \xi_s,s) \cdot [b(y, s) - \xi'_s] m(y, s) \dyds \\
= - \iint_{B_R  \times (0,\tau) } w(y + \xi_s,s) \divv \big(b(y, s) m(y, s)\big) \dyds \\
- \iint_{B_R  \times (0,\tau) } Dw(y + \xi_s,s) \cdot \xi'_s m(y, s) \dyds
\end{multline*}

Adding up the three previous inequalities, one finds the equation for $m$ which is being multiplied by $w(y + \xi_s,s)$, and the two terms $\iint Dw(y + \xi_s,s) \cdot \xi'_s m$ cancel out, so
\begin{multline*}
- \int_{B_R} w(y + \xi_\tau, \tau) m(y, \tau) \dy  + w( \xi_0, 0) + \sigma \iint_{\partial B_R  \times (0,\tau) } w(y + \xi_s,s) Dm  \cdot \nu \dyds + \\ - \ell_0 \iint_{B_R(\xi_s)  \times (0,\tau) } |\tilde b|^{\gamma'}  m(\tilde y - \xi_s, s) \dytds \le \iint_{B_R(\xi_s)  \times (0,\tau) } g m(\tilde y - \xi_s, s) \dytds.
\end{multline*}
After a further change of variables $y = \tilde y - \xi_s$, and the fact that $\xi_\tau = 0$ and $\xi_0 = y_0$ we obtain the assertion.
\end{proof}

\begin{lemma} 
 Let $\Kcal = \iint_{B_R  \times (0,\tau) } |b|^{\gamma'} m \dyds$. Then, for some $C_3$ depending on $\underline h, \alpha, q_0$,
\begin{multline}\label{xestimate}
w(y_0,0) - w(0,0) \le C_3 \Big( \frac{  \Kcal^{1/\gamma}} {\tau^{1/\gamma}} + \frac1{\tau^{\gamma'-1}}  + \sigma^{-\gamma'\frac{N+1}{N+2}} ( \Kcal +\tau^{\alpha_0/2} )  \|g\|_{L^{q_0}(Q_{R+1,\tau})}+ \\ + \frac{\tau^{1/\gamma} \Kcal^{1/\gamma'} }R +  \frac{\tau}{R^2} + (\ell_0-\ell_1)\Kcal \Big),
\end{multline}
where $\ell_i = \frac{h_i(\gamma-1)}{(h_i\gamma)^{\gamma'}}$, $i=1,2$.
\end{lemma}

\begin{proof}
Recall \eqref{wn}, which reads
\begin{multline*}
w(0,0) \ge \ell_1 \Kcal + \iint_{B_R  \times (0,\tau) } g m \dyds + \\ +  \int_{B_R} w(y,\tau) m(y,\tau) \dy- \sigma \iint_{\partial B_R  \times (0,\tau) } w Dm  \cdot \nu \dyds.
\end{multline*}
Subtracting \eqref{wnabove} and the previous equality yields
\begin{align*}
w(y_0,0) - w(0,0) \le 
\ell_0 & \iint_{B_R  \times (0,\tau) } [ |b - \xi'_s|^{\gamma'} - |b|^{\gamma'} ] m \dyds \\ 
 + & \iint_{B_R  \times (0,\tau) } [g(y + \xi_s, s) - g(y,s)] m(y,s) \dyds \\
  - &  \sigma \iint_{\partial B_R  \times (0,\tau) } [w(y + \xi_s,s)-w(y,s)] Dm(y,s)  \cdot \nu \dyds\\
& +(\ell_0-\ell_1) \Kcal.
\end{align*}

To estimate the first term, we make use of the inequality \eqref{Ldiff}, that $|\xi'| = \tau^{-1}$ and the H\"older inequality:
\begin{multline*}
\ell_0\iint_{B_R  \times (0,\tau) } [ |b - \xi'_s|^{\gamma'} - |b|^{\gamma'} ] m \dyds \stackrel{\eqref{Ldiff}}{\le} \\ \le C | \xi'_s| \iint_{B_R  \times (0,\tau) } |b|^{\gamma'-1}  m \dyds + C | \xi'_s |^{\gamma'} \iint_{B_R  \times (0,\tau) }  m \dyds \\
\le \frac C \tau \left( \iint_{B_R  \times (0,\tau) } |b|^{\gamma'}  m \dyds \right)^{\frac1\gamma} \left( \iint_{B_R  \times (0,\tau) } m \dyds \right)^{\frac1{\gamma'}} + \frac C {\tau^{\gamma'}}  \iint_{B_R  \times (0,\tau) }  m \dyds\\
\stackrel{\eqref{massloss}}{\le}
 \frac{ C \Kcal^{1/\gamma}} {\tau^{1/\gamma}} + \frac C {\tau^{\gamma'-1}}.
\end{multline*}

As for the second term, Lemma \eqref{lemmamq0} implies
\begin{multline*}
\iint_{B_R  \times (0,\tau) } [g(y + \xi_s, s) - g(y,s)] m(y,s) \dyds \le 2\|g\|_{L^{q_0}(Q_{R+1,\tau})} \|m\|_{L^{q_0'}(Q_{R,\tau})} \\
\le C \sigma^{-\gamma'\frac{N+1}{N+2}} \|g\|_{L^{q_0}(Q_{R+1,\tau})} \left( \Kcal + \sigma^{\gamma'/2}\tau^{\alpha_0/2} \right)
\end{multline*}

Finally, by the growth assumptions on $w$,
\begin{multline*}
\sigma \iint_{\partial B_R  \times (0,\tau) } [w(y + \xi_s,s)-w(y,s)] (-Dm \cdot \nu) \dyds \le\\ \le 3\sigma \iint_{\partial B_R  \times (0,\tau) }  (-Dm \cdot \nu) \dyds \stackrel{\eqref{boundaryloss}}{\le} \frac{3 \tau^{1/\gamma}}R \Kcal^{1/\gamma'}+  \frac{C \tau}{R^2},
\end{multline*}
and by collecting the three estimate we conclude.
\end{proof}

To conclude the proof of Proposition \ref{mainstima}, it just suffices to collect the estimates \eqref{testimate} and  \eqref{xestimate}.

\section{$\alpha_0$-H\"older regularity}\label{sec:hol}

In this section we prove the $\alpha_0$-H\"older estimates, which are at the base of further regularity results.

\begin{theorem}\label{alpha0reg} Let $\Fcal \subset L^{q_0}(Q)$ be a uniformly integrable set in $L^{q_0}(Q)$ and assume that $\|u\|_{L^\infty(Q)} \le K$. Then there exists $C, z$ depending on $\Fcal, K, Q, h, q_0$ such that
\[
\llbracket u \rrbracket_{\alpha_0, z; Q} \le C.
\] 
\end{theorem}

\textbf{1. Set-up of the contradiction argument.} Fix $z > 0$, which will be specified below. Assume by contradiction that for some sequence $\{f_n\} \subset \Fcal$ and $\{u_n\}$ solving \eqref{hj0},
\[
\llbracket u \rrbracket_{\alpha_0, z; Q} \to \infty \quad \text{as $n \to \infty$}.
\]
Setting $L_n := \llbracket u \rrbracket_{\alpha_0, z; Q} / 2$, we will distinguish two cases:
\begin{center}
\begin{tabular}{l}
\ca \, $2L_n = \llbracket u\rrbracket^x_{\alpha_0; Q} \ge  (z^{-1} \llbracket u\rrbracket^t_{\sfrac{\alpha_0}2; Q} )^{\frac2{\gamma}} $, \\
\cb \, $2L_n =  (z^{-1} \llbracket u\rrbracket^t_{\sfrac{\alpha_0}2; Q})^{\frac2{\gamma}} \ge \llbracket u\rrbracket^x_{\alpha_0; Q} $ .
\end{tabular}
\end{center}
In case \ca, there are sequences $\{(x_n, \bar t_n)\}, \{(\bar x_n, \bar t_n )\} \subset Q$ such that for every $n$, $x_n \neq \bar x_n$ and
\[
L_n  \le \min\{ \dcal((x_n,\bar t_n), \partial Q ), \dcal((\bar x_n,\bar t_n), \partial Q ) \} \frac{|u_n(x_n,\bar t_n) - u(\bar x_n, \bar t_n)|}{|x_n-\bar x_n|^{\alpha_0}} \le 2 L_n.
\]
If we assume w.l.o.g. that $\dcal((x_n,\bar t_n), \partial Q ) \ge \dcal((\bar x_n , \bar t_n), \partial Q )$ and $u(\bar x_n, \bar t_n) = 0$, setting
\[
d_n = d(\bar x_n, \partial \Omega), \qquad M_n = |u_n(x_n,\bar t_n)|, \qquad r_n = |x_n-\bar x_n|,
\]
then the previous inequalities read
\begin{equation}\label{Lna}
L_n  \le \big(d^{\alpha_0}_n + |T-\bar t_n|^{\frac{\alpha_0}\gamma}\big) \frac{M_n}{r_n^{\alpha_0}} \le 2L_n.
\end{equation}

In the other case \cb, there are sequences $\{(\bar x_n, t_n)\}, \{(\bar x_n, \bar t_n )\} \subset Q$ such that for every $n$, $t_n > \bar t_n$ and
\[
L_n  \le z^{-\frac2\gamma }\min\{ \dcal{((\bar x_n,t_n), \partial Q )}, \dcal{(( \bar x_n, \bar t_n), \partial Q )} \} \frac{|u(\bar x_n ,t_n) - u(\bar x_n,\bar t_n)|^{\frac2\gamma}}{|t_n-\bar t_n|^{\frac{\alpha_0}\gamma}}  \le 2L_n.
\]
If we let
\[
d_n = d(\bar x_n, \partial \Omega), \qquad M_n = \frac{|u(\bar x_n,t_n)|}z, \qquad r_n = (t_n-\bar t_n)^{\frac1\gamma}M_n^{\frac{\gamma-1}{\gamma}},
\]
we have (as before $u(\bar x_n, \bar t_n) = 0$)
\begin{equation}\label{Lnb}
L_n  \le   \big(d^{\alpha_0}_n + |T-\bar t_n|^{\frac{\alpha_0}\gamma}\big) \frac{M_n^{\frac2\gamma+\alpha_0\frac{\gamma-1}\gamma}}{r_n^{\alpha_0}} = \big(d^{\alpha_0}_n + |T-\bar t_n|^{\frac{\alpha_0}\gamma}\big) \frac{M_n}{r_n^{\alpha_0}} \le 2L_n.
\end{equation}

In \textit{both} cases, since $M_n \le \|u_n\|_\infty \le K$ and $Q$ is bounded, we infer from \eqref{Lna} and \eqref{Lnb} that, as $n \to \infty$,
\begin{equation}\label{doverr}
\frac{d_n}{r_n} \to +\infty, \qquad r_n \to 0, \qquad
\frac{M_n^{\gamma-1}}{r_n^{\gamma-2}} = \left(\frac{M_n}{r_n^{\alpha_0}}\right)^{\gamma-1} \to +\infty,
\end{equation}
as well as
\begin{align}\label{doverr2}
\begin{split}
(T-\bar t_n)\frac{M_n^{\gamma-1}}{r_n^\gamma} & \ge (T-\bar t_n) \frac{M_n^{\frac\gamma{\alpha_0}}}{r_n^\gamma} K^{-\frac2{\alpha_0}} \to +\infty, \\
\frac{M_n^{\frac{\alpha_0}{\gamma'}} }{r_n^{\alpha_0}} (d^{\alpha_0}_n + |T-\bar t_n|^{\frac{\alpha_0}\gamma}) & \ge \frac{M_n }{r_n^{\alpha_0}} (d^{\alpha_0}_n + |T-\bar t_n|^{\frac{\alpha_0}\gamma}) K^{\frac{\alpha_0}{\gamma'}-1} \to +\infty.
\end{split}
\end{align}

\smallskip

\textbf{2. Scaling.} We now define the blow-up sequence
\begin{multline*}
w_n(y,s) = \frac{1}{M_n}\,u_n\left(\bar x_n + r_n y, \bar t_n + \frac{r_n^\gamma}{M_n^{\gamma-1}} s\right), \\ (y,s) \in \frac{\Omega - \bar x_n}{r_n} \times \left(-\bar t_n\frac{M_n^{\gamma-1}}{r_n^\gamma}, (T-\bar t_n)\frac{M_n^{\gamma-1}}{r_n^\gamma}\right) = Q_n.
\end{multline*}

In case \ca, let $y_0$ be such that $x_n = \bar x_n + r_n y_0$. By the definition of $r_n$, $|y_0| = 1$. Recall also that $u(\bar x_n, \bar t_n) = 0$, so
\begin{equation}\label{anorm}
|w_n(y_0, 0) - w_n(0,0)| = \frac1{M_n}|u_n(x_n,\bar t_n) - u_n(\bar x_n,\bar t_n)| = 1.
\end{equation}
On the other hand, in case \cb, 
\begin{equation}\label{bnorm}
|w_n(0, 1) - w_n(0,0)| = \frac1{M_n}|u_n(\bar x_n, t_n)| = z.
\end{equation}
The sequence $w_n$ solves the following HJ equation
\begin{equation}\label{hjw1}
-\partial_s w_n - \sigma_n \Delta w_n + h\left(\bar x_n + r_n y, \bar t_n + \frac{r_n^\gamma}{M_n^{\gamma-1}} s\right) |Dw_n|^\gamma = g_n \qquad \text{on $Q_n$},
\end{equation}
where
\[
\sigma_n = \frac{r_n^{\gamma-2}}{M_n^{\gamma-1}}, \qquad g_n(y,s) =  \frac{r_n^{\gamma}}{M_n^{\gamma}}f_n\left(\bar x_n + r_n y, \bar t_n + \frac{r_n^\gamma}{M_n^{\gamma-1}} s\right)
\]
Note that $\sigma_n \to 0$ (see \eqref{doverr}).
Moreover, for any $R, \tau > 0$, a straightforward computation yields
\[
\|g_n\|_{L^{q_0}(Q_{R,\tau}) }= \sigma_n^{\gamma'\frac{N+1}{N+2}} \|f_n\|_{L^{q_0}( \widetilde{Q}_{n,R,\tau}) }, \quad \widetilde{Q}_{n,R,\tau} = B_{Rr_n} (\bar x_n) \times (\bar t_n, \bar t_n + r_n^\gamma M_n^{1-\gamma} \tau)
\]
Therefore, since $f_n$ is uniformly integrable in $L^{q_0}(Q)$,
\begin{equation}\label{fnsmall}
\|f_n\|_{L^{q_0}( \widetilde{Q}_{n,R,\tau} )} \to 0 \qquad \text{as $n\to\infty$},
\end{equation}
because $r_n, r_n^\gamma M_n^{1-\gamma} \to 0$ in view of \eqref{doverr}, \eqref{doverr2}. Since $h$ is continuous on $\overline Q$, for any $\eps > 0$ (smaller than $h_0/2$), for $n$ large enough we have on on $Q_n$
\begin{equation}\label{subsupwn}
\begin{gathered}
-\partial_s w_n -  \sigma_n \Delta w_n + (h_n - \eps) |Dw_n|^\gamma \le g_n, \\
-\partial_s w_n -  \sigma_n \Delta w_n + (h_n + \eps) |Dw_n|^\gamma \ge g_n,
\end{gathered}
\end{equation}
where $h_0 \le h_n = h(\bar x_n, \bar t_n) \le h_1$.

In both cases, we have the following control on H\"older seminorms of $w_n$.
\begin{lemma}\label{lemwnholder} For any $R, \tau > 1$, we have for $n$ large enough that $\overline{Q}_{R, \tau}=\overline{B_R} \times [0, \tau] \subset Q_n$ and
\[
\sup_{(y,s), (y',s) \in \overline{Q}_{R, \tau}} \frac{|w_n(y,s) - w_n(y',s)|}{|y-y'|^{\alpha_0}} \le 3, \qquad 
\sup_{(y,s), (y, \bar s) \in \overline{Q}_{R, \tau}} \frac{|w_n(y,s) - w_n(y, s')|}{|s- s'|^{\alpha_0/2}} \le 3^{\frac\gamma2} z.
\]
\end{lemma}

\begin{proof} First,
\[
d(0, \partial \Omega_n) = d\left(0, \frac{\partial \Omega - \bar x_n}{r_n} \right) = \frac1{r_n} d(\bar x_n, \partial \Omega) \to +\infty
\]
by \eqref{doverr}. Moreover, \eqref{doverr2} guarantees that $(T-\bar t_n)\frac{M_n^{\gamma-1}}{r_n^\gamma} \to +\infty$,
therefore $\overline{B_R} \times [0, \tau] \subset Q_n$ whenever $n$ is large.

Let now $(y,s), (y',s) \in \overline{Q}_{R, \tau}$ and assume that $\dcal((y',s), \partial Q_n ) \ge \dcal((y , s), \partial Q_n )$. Then, writing
\[
(x,t)=\left(\bar x_n + r_n y, \bar t_n + \frac{r_n^\gamma}{M_n^{\gamma-1}} s\right), \quad (x',t)=\left(\bar x_n + r_n y', \bar t_n + \frac{r_n^\gamma}{M_n^{\gamma-1}} s\right)
\]
we get
\begin{equation}\label{eq001}
\frac{|u_n(x,t) - u_n(x',t)|}{|x-x'|^{\alpha_0}} \le \frac{[u]^x_{\alpha_0; Q}}{d^{\alpha_0}(x, \partial \Omega) + |T-t|^{\frac{\alpha_0}\gamma}} \le \frac{2L_n}{d^{\alpha_0}(x, \partial \Omega) + |T-t|^{\frac{\alpha_0}\gamma}}.
\end{equation}
Now, 
\begin{align}\label{eq0015}
\begin{split}
d^{\alpha_0}_n + (T-\bar t_n)^{\frac{\alpha_0}\gamma} & \le d^{\alpha_0}(x, \partial \Omega) + (T-t)^{\frac{\alpha_0}\gamma} + |x-\bar x_n|^{\alpha_0} + |t-\bar t_n|^{\frac{\alpha_0}\gamma} \\
& \le d^{\alpha_0}(x, \partial \Omega) + (T-t)^{\frac{\alpha_0}\gamma} + r_n^{\alpha_0} R^{\alpha_0} + \frac{r_n^{\alpha_0}}{M_n^{\frac{\alpha_0}{\gamma'}}} \tau^{\frac{\alpha_0}\gamma} \\
& \le d^{\alpha_0}(x, \partial \Omega) + (T-t)^{\frac{\alpha_0}\gamma} + C \frac{r_n^{\alpha_0}}{M_n^{\frac{\alpha_0}{\gamma'}}} ,
\end{split}
\end{align}
for some $C$ depending on $K, R, \tau$.
Therefore,
\begin{multline*}
\frac{|w_n(y,s) - w_n(y',s)|}{|y-y'|^{\alpha_0}} = \frac{|u_n(x,t) - u_n(x',t)|}{|x-x'|^{\alpha_0}} \frac{r_n^{\alpha_0}}{M_n} \stackrel{\eqref{Lna},\eqref{eq001}}{\le}
\frac{2L_n}{d^{\alpha_0}(x, \partial \Omega) + |T-t|^{\frac{\alpha_0}\gamma}} \frac{d^{\alpha_0}_n + |T-\bar t_n|^{\frac{\alpha_0}\gamma}}{L_n} \\
\stackrel{\eqref{eq0015}}{\le} 2 \left(1- C \frac{r_n^{\alpha_0}}{M_n^{\frac{\alpha_0}{\gamma'}} (d^{\alpha_0}_n + |T-\bar t_n|^{\frac{\alpha_0}\gamma}) }\right)^{-1} \le 3,
\end{multline*}
the last inequality being true for large $n$.

Similarly, $(y,s), (y,s') \in \overline{Q}_{R, \tau}$ be such that $\dcal((y,s'), \partial Q_n ) \ge \dcal((y , s), \partial Q_n )$. Then, for
\[
(x,t)=\left(\bar x_n + r_n y, \bar t_n + \frac{r_n^\gamma}{M_n^{\gamma-1}} s\right), \quad (x,t')=\left(\bar x_n + r_n y, \bar t_n + \frac{r_n^\gamma}{M_n^{\gamma-1}} s'\right)
\]
we have
\begin{equation}\label{eq002}
\frac{|u_n(x,t) - u_n(x,t')|}{|t-t'|^{\frac{\alpha_0}2}} \le \frac{[u_n]^t_{\frac{\alpha_0}2; Q}}{[d^{\alpha_0}(x, \partial \Omega) + |T-t|^{\frac{\alpha_0}\gamma}]^{\frac\gamma2}}\le
z2^{\frac\gamma2}\frac{L_n^{\frac\gamma2}}{[d^{\alpha_0}(x, \partial \Omega) + |T-t|^{\frac{\alpha_0}\gamma}]^{\frac\gamma2}}.
\end{equation}
Thus,
\begin{multline*}
\frac{|w_n(y,s) - w_n(y,s')|}{|s-s'|^{\frac{\alpha_0} 2}} = \frac{|u_n(x,t) - u_n(x,t')|}{|t-t'|^{\frac{\alpha_0}2}} \frac{r_n^{\alpha_0 \frac\gamma2}}{M_n^{\frac\gamma2}}\\  \stackrel{\eqref{Lna},\eqref{eq002}}{\le}
\frac{2^{\frac\gamma2} z L^{\frac\gamma2}_n}{ [d^{\alpha_0}(x, \partial \Omega) + |T-t|^{\frac{\alpha_0}\gamma}]^{\frac\gamma2} } \frac{[d^{\alpha_0}_n + |T-\bar t_n|^{\frac{\alpha_0}\gamma}]^ \frac\gamma2}{L_n^ \frac\gamma2} \\
\stackrel{\eqref{eq0015}}{\le} 2^{\frac\gamma2} z \left(1- C \frac{r_n^{\alpha_0}}{M_n^{\frac{\alpha_0}{\gamma'}} (d^{\alpha_0}_n + |T-\bar t_n|^{\frac{\alpha_0}\gamma}) }\right)^{-\frac\gamma2} \le 3^{\frac\gamma2} z.
\end{multline*}

\end{proof}

\textbf{3. Conclusion. } We now have to choose $z, \tau, R, \eps$, and pick $n$ large enough so that a contradiction is reached. To avoid circular dependencies, the general idea will be to choose first $z$ large, then $\tau$, then $R$ large enough, and finally $\eps$ small. The smallness of $\|f_n\|$ will also play a crucial role.

In view of Lemma \ref{lemwnholder}, we can employ Proposition \ref{mainstima} with $\alpha = \alpha_0$ (note that in this case $\frac{\alpha_0}{\gamma-\alpha_0(\gamma-1)} = \frac{\alpha_0}2$ and $\sigma_n^{-\gamma'\frac{N+1}{N+2}} \|g_n\|_{L^{q_0}(Q_{R,\tau}) } =  \|f_n\|_{L^{q_0}( \widetilde{Q}_{n,R,\tau}) }$), which yields
\begin{equation}  \label{test2}
 |w_n(0,0)-w_n(0,\tau)| +  \Kcal  \le C_2\left(\tau^{\frac{\alpha_0}2} +  \tau \frac{R^{\alpha_0} + \tau^{\frac{\alpha_0}2}}{R}z  \right)
 \end{equation}
 and
 \begin{multline}\label{xest2}
|w_n(y_0,0) - w_n(0,0)| \le C_3 \Big( \frac{  \Kcal^{1/\gamma}} {\tau^{1/\gamma}} + \frac1{\tau^{\gamma'-1}}  +  ( \Kcal + \tau^{\alpha_0/2} ) \|f_n\|_{L^{q_0}( \widetilde{Q}_{n,R+2,\tau})} \\ +  \frac{\tau^{1/\gamma} \Kcal^{1/\gamma'} }R +  \frac{\tau}{R^2} + \eta_n \Kcal \Big), 
 \end{multline}
provided that 
\[
\overline{B_{R+2}} \times [0, \tau] \subset Q_n, \qquad \|f_n\|_{L^{q_0}( \widetilde{Q}_{n,R,\tau}) } \le f^0, \qquad  R^2 \ge \sigma_n \tau, \qquad z \frac{R^{\alpha_0} + \tau^{\frac{\alpha_0}2}}{R} \le c_1,
\]
where, since $h_n - \eps \ge h_0 / 2$,
\[
\eta_n = \frac{(\gamma-1)}{\gamma^{\gamma'}}\left(\frac1{(h_n-\eps)^{\gamma'}}-\frac1{(h_n+\eps)^{\gamma'}} \right) \le C_4 \eps.
\]
Here $f^0, c_1, C_2, C_3$ are universal constants. The first and the second condition are always satisfied for large $n$ (see Lemma \ref{lemwnholder} and \eqref{fnsmall}). Note that the absolute value of $|w_n(y_0,0) - w_n(0,0)|$ appears in \eqref{xest2}; this is a consequence of \eqref{xest0}, and \eqref{xest0} again if one exchanges the roles of $0$ and $y_0$ (which can be done eventually enlarging the domain from $Q_{R, \tau}$ to $Q_{R+1, \tau}$). The third and the fourth condition will be verified below.

\smallskip
We first aim to rule out case \cb. To do so, we fix
\[
z = 4C_2, \qquad \tau = 1
\]
and $R > 1$ large enough so that
\[
 \quad 4C_2 \frac{R^{\alpha_0} + 1}{R} \le \min\{1,c_1\},
\]
hence \eqref{test2} applies, yielding
\[
|w_n(0,0)-w_n(0,1)| \le C_2\left(1 +  4C_2 \frac{R^{\alpha_0} + 1}{R}  \right) \le 2C_2 = \frac{z}2.
\]
For large $n$, $Q_{R,1} \subset \tilde Q_n$, hence \eqref{bnorm} is impossible.

\smallskip

We now treat case \ca, with the same choice of $z$ as before, but possibly with larger $\tau, R$. First, pick $\tau \ge 1$ so that
\[
C_2 \tau^{\frac{\alpha_0}2}\le \left(\frac1{\den C_3}\right)^{\gamma}\tau, \qquad \frac1{\tau^{\gamma'-1}} \le \frac1{\den C_3},
\]
and $R$ large such that
\[
C_2\left( \tau \frac{R^{\alpha_0} + \tau^{\frac{\alpha_0}2}}{R}z  \right) \le \left(\frac1{\den C_3}\right)^{\gamma} \tau, \quad \left(\ 2 \left(\frac1{\den C_3}\right)^{\gamma} \tau \right)^{1/\gamma'} \frac{\tau^{1/\gamma}  }R \le \frac1{\den C_3}, \quad \frac{\tau}{R^2} \le \frac1{\den C_3}.
\]
Then, $n$ be large so that
\[
\left(2 \left(\frac1{\den C_3}\right)^{\gamma} \tau + \tau^{\alpha_0/2} \right)\|f_n\|_{L^{q_0}( \widetilde{Q}_{n,R+2,\tau}) } \le \frac1{\den C_3}, \qquad \text{and $Q_{R+2,\tau} \subset Q_n$}.
\]
With these choices, \eqref{test2} reads
\[
\Kcal  \le C_2\left(\tau^{\frac{\alpha_0}2}  +  \tau \frac{R^{\alpha_0} + \tau^{\frac{\alpha_0}2}}{R}z  \right) \le 2 \left(\frac1{\den C_3}\right)^{\gamma} \tau,
\]
whence
\[
\frac{\Kcal^{1/\gamma}} {\tau^{1/\gamma}} \le \frac{2^{1/\gamma}}{\den C_3}, \qquad \frac{\tau^{1/\gamma} \Kcal^{1/\gamma'} }R \le \frac1{\den C_3}, \qquad ( \Kcal + \tau^{\alpha_0/2} ) \|f_n\|_{L^{q_0}( \widetilde{Q}_{n,R+2,\tau}) } \le \frac1{\den C_3}.
\]
We can eventually pick $\eps > 0$ small (and consequently increase $n$, if needed) so that
\[
\eta_n \Kcal \le 2 C_4 \left(\frac1{\den C_3}\right)^{\gamma} \tau \eps \le \frac1{\den C_3}.
\]
Finally, we plug the previous inequalities into \eqref{xest2} to get
\[
|w_n(y_0,0) - w_n(0,0)| \le  \frac{2^{1/\gamma} + 5}{\den} \le \frac7{\den},
\]
contradicting \eqref{anorm}, so case \ca \,  is also impossible, and this proves Theorem \ref{alpha0reg}.

\begin{remark}[On the role of $h$] The constant $C$ in Theorem \ref{alpha0reg} depends on $h$ just through its lower and upper bounds $h_0, h_1$, and its modulus of continuity on $\overline Q$.
\end{remark}

\begin{remark}[On the assumptions]\label{remass} The proof of Theorem \ref{alpha0reg} can be carried out analogously for solutions of 
\[
-\partial_t u - \Delta D^2 u + H(x, t, Du) = 0
\]
under the following assumption on $H$: for some $0 < h_0 \le h_1$, for all $\eps > 0$ there exists $\delta = \delta (\eps)$ such that if $|Q| < \delta$, then
\[
(h_Q-\eps)|\xi|^\gamma - f_Q(x,t) \le H(x,t,\xi) \le (h_Q+\eps)|\xi|^\gamma + f_Q(x,t) \qquad \text{on $Q \times \R^N$}
\]
for some $h_0 \le h_Q \le h_1$ and $\|f_Q\|_{L^{q_0}(Q)} \le \eps$. Indeed, this is enough to guarantee \eqref{fnsmall}, \eqref{subsupwn}, and the rest of the proof remains unchanged. For example, 
\[
H(x,t,\xi) = h|\xi|^\gamma + b(x,t) \cdot \xi,
\]
where $b$ varies in a set of uniformly $L^{N+2}$-integrable function, satisfies the previous assumption. Note finally that we actually need \eqref{subsupwn} for \textit{some} $\eps$ small enough, hence discontinuous $h$ are allowed, provided that the local oscillation is smaller than some universal constant $\eps_0 > 0$.
\end{remark}

\section{A Liouville theorem}\label{sec:lio}

Let $q > q_0$, $v \in W_{q,loc}^{2,1}(\R^N \times (0,\infty))$ be a strong solution of 
\begin{equation}\label{hjv}
-\partial_t v - \Delta v + h |Dv|^\gamma = 0 \qquad \text{on $ \R^N \times (0,\infty)$},
\end{equation}
$h \neq 0$. The purpose of this section is to prove the following Liouville-type theorem for ``far future'' solutions  $v$ (which would be just ``ancient'' solutions if the equation was forward in time). Note that a standard bootstrap procedure shows that $v$ is actually smooth, up to $C^\infty$.

\begin{theorem}\label{liouville} Suppose that for some $0< \alpha < 1$,
\[
\sup_{(x,t) \neq (x',t') \in \R^N \times (0,\infty) } \frac{|v(x,t) - v(x',t')|}{|x-x'|^\alpha + |t-t'|^{\alpha/2}} < \infty.
\]
Then, $v$ is constant on $\R^N \times (0,\infty)$.
\end{theorem}

\begin{proof} The linear case $h = 0$ is proved in \cite[Theorem 1.2(b)]{SZ06}. Otherwise, w.l.o.g. we may assume that $h > 0$ (if $h < 0$ it is sufficient to replace $v$ by $-v$). We may also assume that 
\[
\sup_{(y,s), (y',s) \in  \R^N \times (0, \infty)} \frac{|v(y,s) - v(y',s)|}{|y-y'|^\alpha} \le 3, \qquad 
\sup_{(y,s), (y, \bar s) \in  \R^N \times (0, \infty)} \frac{|v(y,s) - v(y, s')|}{|s- s'|^{\alpha/2}} \le 3,
\]
which can be always verified after dividing $v$ by a suitable positive constant (and changing $h$ accordingly).

Then, using Proposition \ref{mainstima} with $g \equiv 0$ and $h_0 = h_1 = h$,
\begin{align*}
& v(x_0,0) - v(0,0) \le  C_3 \left( \frac{  \Kcal^{1/\gamma}} {\tau^{1/\gamma}} + \frac1{\tau^{\gamma'-1}}  +  \frac{\tau^{1/\gamma} \Kcal^{1/\gamma'} }R +  \frac{\tau}{R^2} \right), \\
& \qquad\qquad \Kcal  \le C_2\left( \tau^{\frac\alpha2} + \tau^{\frac{\alpha_0}2}  +  \tau^{\frac{\alpha}{\gamma-\alpha(\gamma-1)}} +  \tau \frac{R^\alpha + \tau^{\frac\alpha2}}{R}  \right),
\end{align*}
for any $R \ge R_0 = R_0(\tau)$ and $|x_0| \le 1$. If we let $R \to +\infty$, we obtain
\[
v(x_0,0) - v(0,0) \le  C_3 \left( \left(\frac{ \tau^{\frac\alpha2} + \tau^{\frac{\alpha_0}2}  +  \tau^{\frac{\alpha}{\gamma-\alpha(\gamma-1)}}   }{\tau}\right)^{1/\gamma} + \frac1{\tau^{\gamma'-1}} \right).
\]
Since $\alpha/2, \alpha_0/2, \alpha/(\gamma-\alpha(\gamma-1)) < 1$, as $\tau \to \infty$ we conclude that $v(x_0,0) - v(0,0) \le 0$ for any $|x_0| \le 1$. By exchanging the roles of $x_0$ and $0$ we get that $v(\cdot, 0)$ is constant. The same argument applies to $v(\cdot, t)$ for any $t > 0$, which allows to conclude that $v(\cdot, t)$ is constant for any $t > 0$. Since now $\partial_t v \equiv 0$ on $\R^N \times (0, \infty)$, the conclusion is reached.
\end{proof}

\begin{remark}[On the behavior of $v$]\label{betterliou} We note that the main requirement on $v$ is its behavior as $(|x|, t)$ goes to infinity. We could indeed get the same result assuming only, for $0 < \alpha < 1$ and $\beta > 0$,
\[
\sup_{(x,t), (x',t') \in \R^N \times (0,\infty) } \frac{|v(x,t) - v(x',t')|}{|x-x'|^\alpha + |t-t'|^\beta + 1} < \infty,
\]
using the arguments developed in the previous section. Note that the statement would be false if $\alpha = 1$ was allowed, just consider $v(x,t) = ht + x$.
\end{remark}

\begin{remark}[On the role of the diffusion] Since the right-hand side in \eqref{hjv} is zero, there is no real need of the regularizing effect of a nondegenerate diffusion. In other words, we can replace $\Delta v$ by $\tr A D^2 v$, $A \ge 0$, and the proof of the Liouville theorem works, at least \textit{formally}. To make it rigorous, one has then to be careful on the notion of solution for \eqref{hjv}, and the dual Fokker-Planck equation which is used in Proposition \ref{mainstima}.
\end{remark}

\section{$\alpha$-H\"older regularity}\label{sec:hol2}

\begin{theorem}\label{alphareg} Assume that $\|u\|_{L^\infty(Q)}, [u]_{\alpha_0; Q}, \|f\|_{L^q(Q)} \le K$. Then, there exists $C$ depending on $K, Q, h, q, N$ such that
\[
[u]^{\alpha-\alpha_0}_{\alpha; Q} \le C,
\] 
where $\alpha = 2 - \frac{N+2}q$.
\end{theorem}

\textbf{1. Set-up of the contradiction argument.} Assume by contradiction that for some sequence $\{f_n\}$ and $\{u_n\}$ solving \eqref{hj0},
\[
[u]^{\alpha-\alpha_0}_{\alpha; Q} \to \infty \quad \text{as $n \to \infty$}.
\]
Setting $L_n := [u]^{\alpha-\alpha_0}_{\alpha; Q} / 2$ there are sequences $\{(x_n,  t_n)\}, \{(\bar x_n, \bar t_n )\} \subset Q$ such that for every $n$, $(x_n, t_n) \neq (\bar x_n, \bar t_n)$ and
\[
L_n  \le [\min\{ d((x,t), \partial Q ), d((\bar x, \bar t), \partial Q ) \} ]^{\alpha-\alpha_0} \frac{|u(x,t) - u(\bar x, \bar t)|}{(|x-\bar x| + |t-\bar t|^{\frac12})^{\alpha}} \le 2 L_n.
\]
If we assume w.l.o.g. that $d((x_n,  t_n), \partial Q ) \ge d((\bar x_n , \bar t_n), \partial Q )$ and $u(\bar x_n, \bar t_n) = 0$, setting
\[
d_n = d((\bar x_n , \bar t_n), \partial Q ), \qquad M_n = |u_n(x_n, t_n)|, \qquad r_n = |x_n-\bar x_n| + |t_n-\bar t_n|^{\frac 12}.
\]
the previous inequality reads
\begin{equation}\label{Ln2}
L_n  \le d^{\alpha-\alpha_0}_n \frac{M_n}{r_n^\alpha} \le 2L_n.
\end{equation}

Since $\frac{M_n}{r_n^{\alpha_0}} \le K$ and $Q$ is bounded, we infer from \eqref{Ln2} that, as $n \to \infty$,
\begin{equation}\label{doverr2b}
\frac{d_n}{r_n} \to +\infty, \qquad r_n \to 0, \qquad \frac{r_n^\alpha}{M_n} \to 0. 
\end{equation}

\smallskip

\textbf{2. Scaling.} We now define the blow-up sequence
\[
w_n(y,s) = \frac{1}{M_n}\,u_n\left(\bar x_n + r_n y, \bar t_n + r_n^2 s\right), \qquad (y,s) \in \frac{\Omega - \bar x_n}{r_n} \times \left(-\frac{\bar t_n}{r_n^2}, \frac{T-\bar t_n}{r_n^2}\right) = Q_n.
\]

Let $(y_0,s_0)$ be such that $(x_n, t_n) = (\bar x_n + r_n y_0, \bar t_n + r_n^2 s_0)$. By the definition of $r_n$, $|y_0| + |s_0|^{1/2}= 1$. Recall also that $u(\bar x_n, \bar t_n) = 0$, so
\begin{equation}\label{norm2}
|w_n(y_0, s_0) - w_n(0,0)| = \frac1{M_n}|u_n(x_n, t_n) - u_n(\bar x_n,\bar t_n)| = 1.
\end{equation}
The sequence $w_n$ solves the following HJ equation
\begin{equation}\label{hjw2}
-\partial_s w_n - \Delta w_n +  \theta_n h(\bar x_n + r_n y, \bar t_n + r_n^2 s) |Dw_n|^\gamma = g_n \qquad \text{on $Q_n$},
\end{equation}
where
\[
\theta_n = \frac{M_n^{\gamma-1}}{r_n^{\gamma-2}}, \qquad g_n(y,s) =  \frac{r_n^2}{M_n}f_n\left(\bar x_n + r_n y, \bar t_n + r_n^2 s\right)
\]
Note that by the assumptions,
\begin{equation}\label{thetabound}
\theta_n^{\frac1{\gamma-1}} = \frac{|u_n(x_n, t_n)- |u_n(\bar x_n, \bar t_n)|}{(|x_n-\bar x_n| + |t_n-\bar t_n|^{\frac 12})^{\alpha_0}} \le K.
\end{equation}
Moreover, for any $R, \tau > 0$, a straightforward computation yields
\[
\|g_n\|_{L^q(Q_{R,\tau}) }= \frac{r_n^{2-\frac{N+2}{q}}}{M_n}\|f_n\|_{L^q( \widetilde{Q}_{n,R,\tau}) }, \quad \widetilde{Q}_{n,R,\tau} = B_{Rr_n} (\bar x_n) \times (\bar t_n, \bar t_n + r_n^2 \tau).
\]
Therefore, since $\|f_n\|_q \le K$ and $\frac{r^{2-\frac{N+2}{q}}}{M_n} = \frac{r^{\alpha}}{M_n} \to 0$ by \eqref{doverr2b}, then
\begin{equation}\label{gnsmall}
\|g_n\|_{L^q(Q_{R,\tau}) } \to 0 \qquad \text{as $n\to\infty$}.
\end{equation}

\begin{lemma}\label{lemwnholder2} For any $R>1$, we have for $n$ large enough that $\overline{Q}_{R}=\overline{B_R} \times [0, R^2] \subset Q_n$ and
\[
\sup_{(y,s), (y',s') \in \overline{Q}_{R}} \frac{|w_n(y,s) - w_n(y',s')|}{(|y-y'|+|s-s'|^{\frac12})^\alpha} \le 3.
\]
\end{lemma}

\begin{proof} First,
\[
d((0,0), \partial^+ Q_n) = d\left(0, \frac{\partial \Omega - \bar x_n}{r_n}  \right) +  \frac{(T-\bar t_n)^{1/2}}{r_n} = \frac{ d_n }{r_n} \to +\infty
\]
by \eqref{doverr2b}. 

Let now $(y,s), (y',s') \in \overline{Q}_{R}$ and assume that $d((y',s'), \partial^+ Q_n ) \ge d((y , s), \partial^+ Q_n )$. Then, writing
\[
(x,t)=\left(\bar x_n + r_n y, \bar t_n + r_n^2 s\right), \quad (x',t')=\left(\bar x_n + r_n y', \bar t_n + r_n^2 s'\right)
\]
we get
\begin{equation}\label{eq004}
\frac{|u_n(x,t) - u_n(x',t')|}{(|x-x'| + |t-t'|^{1/2})^\alpha} \le \frac{[u]^x_{\alpha; Q}}{d^{\alpha-\alpha_0}((x,t), \partial^+ Q)} \le \frac{2L_n}{ d^{\alpha-\alpha_0}((x,t), \partial^+ Q) }.
\end{equation}
Now, 
\begin{align}\label{eq0016}
\begin{split}
d_n = d((\bar x_n, \bar t_n), \partial^+ Q) & \le d(x, \partial \Omega) + (T-t)^{\frac12} + |x-\bar x_n| + |t-\bar t_n|^{\frac12} \\
& \le d(x, \partial \Omega) + (T-t)^{\frac12} + 2 r_n R \\
& = d((x, t), \partial^+ Q) + 2 r_n R.
\end{split}
\end{align}
Therefore,
\begin{multline*}
\frac{|w_n(y,s) - w_n(y',s')|}{(|y-y'| + |s-s'|^{1/2})^\alpha} = \frac{|u_n(x,t) - u_n(x',t')|}{(|x-x'| + |t-t'|^{1/2})^\alpha} \frac{r_n^\alpha}{M_n} \stackrel{\eqref{Ln2},\eqref{eq004}}{\le}
\frac{2L_n}{ d^{\alpha-\alpha_0}((x,t), \partial^+ Q) } \frac{d_n^{\alpha-\alpha_0}}{L_n} \\
\stackrel{\eqref{eq0016}}{\le} 2 \left(1- 2R \frac{r_n}{d_n}\right)^{-(\alpha-\alpha_0)} \le 3,
\end{multline*}
the last inequality being true for large $n$.
\end{proof}

\begin{lemma}\label{lemwnbound} For any $R>1$, there exists $C$ such that
\[
\|\partial_t w_n\|_{q,Q_{R}} + \|D^2 w_n\|_{q;Q_{R}} \le C. 
\]
for $n$ large enough.
\end{lemma}

\begin{proof} First, $n$ need to be large so that $Q_{2R} \subset Q_n$ (which is possible in view of the previous lemma). Moreover, by \eqref{gnsmall}, again for large $n$ one has $\|g_n\|_{L^q(Q_{2R}) } \le 1$. Since $w_n$ solves the HJ equation \eqref{hjw2},
\[
| -\partial_s w_n - \Delta w_n| = | -  \theta_n h(\bar x_n + r_n y, \bar t_n + r_n^2 s) |Dw_n|^\gamma + g_n| \stackrel{\eqref{thetabound}}{\le} h_1 K^{\gamma-1}|Dw_n|^\gamma + 1,
\]
the conclusion follows by invoking Proposition \ref{prop:holdertoW2q}.
\end{proof}

\textbf{3. The limit problem.} We now pass to the limit in the equation \eqref{hjw2}. Limits below are up to subsequences, and they are constructed via a standard diagonal procedure. First, by Lemmata \ref{lemwnholder2} and \ref{lemwnbound}, $w_n$ has (local) weak limits in $W^{2,1}_q$ on $R^N\times(0,\infty)$.

We may assume that $\bar x_n$ and $\theta_n$ converge to $\bar x_\infty$ and $\theta_\infty$ respectively. Note that standard parabolic embeddings guarantee that $Dw_n$ converge (locally) strongly in $L^p$ for any $p < \frac{(N+2)q}{N+2-q}$. Therefore, as $\gamma q <  \frac{(N+2)q}{N+2-q}$,
\[
\theta_n h(\bar x_n + r_n y, \bar t_n + r_n^2 s) |Dw_n|^\gamma \to \theta_\infty h(\bar x_\infty, \bar t_\infty) |Dw|^\gamma
\]
locally strongly in $L^q$, and the same holds for $g_n$ (see \eqref{gnsmall}). Finally, $w_n$ converges locally uniformly to some $w$ which cannot be constant on $Q_2$ in view of \eqref{norm2}.

Hence, the limit $w \in W^{2,1}_{q,loc}(\R^N \times (0,\infty))$ is a nonconstant, strong solution of
\[
-\partial_s w - \Delta w + \theta_\infty h(\bar x_\infty, \bar t_\infty) |Dw|^\gamma = 0 \qquad \text{on $\R^N \times (0,\infty)$}.
\]
It also satisfies
\[
\sup_{(y,s), (y',s')} \frac{|w(y,s) - w(y',s')|}{(|y-y'|+|s-s'|^{\frac12})^\alpha} \le 3,
\]
hence the Liouville Theorem \ref{liouville} applies, contradicting the fact that $w$ cannot be constant.

\section{Proof of Theorem \ref{maxreg} and final remarks}\label{sec:max}

We now conclude with the proof the main theorem on maximal regularity, and list a few remarks.

\begin{proof}[Proof of Theorem \ref{maxreg}] Fix any $Q'', Q'''$ such that $Q' \ssubset Q'' \ssubset Q''' \ssubset Q$. Note that since $q > q_0$ and $Q$ is bounded, functions satisfying $\|f\|_{q; Q}  \le K$ belong to a uniformly integrable set of $L^{q_0}(Q)$. Therefore, Theorem \ref{alpha0reg} yields the existence of $C$ such that
\[
\frac{|u(x,t) - u(\bar x,t)|}{|x-\bar x|^\alpha} + \left( \frac{|u(x,t) - u(x,\bar t)|}{|t-\bar t|^{\sfrac{\alpha}2}} \right)^{\sfrac2\gamma} \le C
\]
for any $(x,t), (x, \bar t), (\bar x, t) \in Q'''$. Hence Theorem \ref{alphareg} applies on $Q'''$, and for some $C' > 0$,
\[
[u]_{\alpha; Q''} \le C'.
\]
Finally, one achieves the desired estimate by Proposition \ref{prop:holdertoW2q}.
\end{proof}

\begin{remark}[$q \ge  N+2$ and Lipschitz estimates]\label{beyondNplus2} If $q \ge  N+2$, then the same conclusion of Theorem \ref{maxreg} holds using a standard bootstrap argument. Choose indeed any $(N+2) \frac{\gamma+1}\gamma < \tilde q < N+2$, and apply Theorem \ref{maxreg} to control $u$ in $W^{2,1}_{\tilde q}$. Therefore, by Sobolev embeddings one has a bound of $|Du|^\gamma$ in $L^p$, where $1/p = 1/\tilde{q} - 1/(N+2)$, and $p > N+2$. Hence, linear parabolic regularity allows to control $u$ in $W^{2,1}_{\min\{p,q\}}$ on a smaller set, which in turn yields \textit{Lipschitz} estimates. A further application of parabolic regularity gives finally bounds in $W^{2,1}_{q}$.

Note that this extends Lipschitz estimates to the full range $q > N+2$. Previously, these were known \cite{CG2} up to $q > \frac{N+2}{2(\gamma'-1)}$ (when $\gamma > 3$).
\end{remark}

\begin{remark}[On the sign of $f$]\label{fsign} Since $f$ can be unbounded, its sign plays a major role. If one has in mind that $f$ appears in the running cost of an optimal control problem, then optimal trajectories are pushed away by \textit{positive} singular points of $f$, while they are strongly attracted by \textit{negative} ones. Heuristically, the former affects the value function $u$ less than the latter. If one assumes $f \ge 0$, then it shown in \cite{CSil} that
\[
f \in L^q, \qquad q >  1 + \frac{N}{\gamma}
\]
is sufficient for $u$ to be H\"older continuous, even in first-order problems. Note that if $\gamma > 2$, then $1 + \frac{N}{\gamma}$ is always smaller than our threshold $q_0$, and it even improves as $\gamma$ increases.

Here, we do not assume any sign condition on $f$. It is not clear whether the critical integrability $q_0$ is related just to the behavior of the negative part of $f$, but we believe that for some $f \notin L^{q_0}$, $u$ is bounded from below. This is strictly related to exponent $q_0'$ appearing in \eqref{mLq0}, which is ``natural'' by invariance properties of the Fokker-Planck equation, but its sharpness is not clear, as no counterexamples below $q_0'$ are available yet.

Not that if no sign condition is imposed on $f$, then our results hold also for the so-called ``repulsive'' case $h < 0$. By the methods developed here, it is not clear if they can be extended to arbitrary (possibly vanishing) $h$.
\end{remark}


\appendix
\section{Miscellaneous estimates}\label{onthefp}

We first start with some estimates on solutions to Fokker-Planck equations. Let $\sigma > 0$, $b$ a measurable vector field such that $b \in L^{N+2}(Q_{R,\tau})$ and $\divv b \in L^{\frac{N+2}{2}}(Q_{R,\tau})$. Below, $m$ be a  solution of
\begin{equation}\label{FP}
\begin{cases}
\partial_t m - \sigma \Delta m - \divv(b m) = 0& \text{on $B_R \times (0,\tau)$} \\
m(x,t) = 0 & \text{on $\partial B_R \times (0,\tau)$} \\
m|_{t=0}= \delta_0
\end{cases}
\end{equation}
Such a solution can be constructed by approximation, that is, taking limits of classical solutions with smooth approximating drifts $b_k$ and smooth initial data $m_k(0)$. Due to the singular nature of $m(0)$, $m \in C([0, \tau]; \Mcal(\Omega))$ is a distributional solution. Furthermore, $m$ is on $Q_{R,\tau}$ a strong solution, since classical parabolic Calder\'on-Zygmund estimates apply away from $t=0$. In particular, $m \in W^{2,1}_{\tilde q}(B_R \times (\tau',\tau))$ for any $0< \tau' < \tau$ and $1 < \tilde q < (N+2)/2$ (see for example \cite[Theorem IV.9.1]{LSU}). Moreover, $m \in L^{q_0}(Q_{R,\tau})$, since the estimate \eqref{mLq0} below ``sees'' only the $L^1$-norm of the initial datum. Finally, the trace of $m(t)$ on $\partial B_R$ is well defined for a.e. $t$, and it belongs  to $W^{2-1/\tilde q}_{\tilde q}(\partial B_R)$.

\smallskip
We list several properties of $m$. First, $- Dm \cdot \nu \ge 0$ a.e. on $\partial B_R \times (0, \tau)$. Indeed, $m$ is a (limit of) non-negative solution vanishing on the lateral boundary $\partial B_R  \times (0, \tau)$, hence $Dm \cdot \zeta \le 0$ on $\partial B_R \times (0, \tau)$ for any $\zeta \cdot \nu > 0$, and clearly $\nu \cdot \nu >0$. Moreover, for a.e. $s \in (0, \tau]$,
\begin{equation}\label{massloss}
\int_{B_R} m(x,s) \dx \le \int_{B_R} m(x,s) \dx - \sigma \iint_{ \partial B_R \times (0,s) } Dm \cdot \nu \dxdt = 1.
\end{equation}
Note that $m$ is the density of a killed Brownian particle starting from $x = 0$ at time $t = 0$, 
on the domain $B_R$). The term $- \sigma \iint Dm \cdot \nu$ represents the probability that the particle has reached $\partial B_R$ before $t  = \tau$. We will estimate below such probability, as a corollary of the following lemma.

\begin{lemma}\label{bdglemma} Assume that  $|b| m \in L^1(Q_{R,\tau})$, and let  $\alpha \in (0,1)$. Then, for some $C$ depending on $\alpha$,
\[
\int_{B_R} |x|^\alpha m(x,\tau) \dx \le  C \left( \iint_{B_R  \times (0,\tau) } |b| m \dxdt  \right)^\alpha+ C \left(\sigma \tau \right)^{\alpha/2}.
\]
\end{lemma}

\begin{proof} Let $\chi \in C^\infty([0,+\infty); [0,+\infty) )$ be such that $\chi(r) \equiv 1$ for all $r \ge 1$ and $\chi(0)=0$. For $0 < y \le R$, test the equation for $m$ by $\chi_y(x) = \chi(|x|/y)$ to get
\begin{multline*}
\int_{B_R} \chi_y(x) m(x, \tau) \dx - \sigma \iint_{\partial B_R  \times (0,\tau) }  \chi_y Dm \cdot \nu \dxdt = \\ =
\sigma \iint_{ B_R  \times (0,\tau) } m \Delta \chi_y \dxdt - \iint_{ B_R  \times (0,\tau) }m b \cdot D \chi_y  \dxdt,
\end{multline*}
whence
\begin{equation}\label{eqzz} 
\int_{|x| \ge y} m(x, \tau) \dx - \sigma \iint_{\partial B_R  \times (0,\tau) }  Dm \cdot \nu \dxdt
\le  \frac{C\sigma \tau }{y^2} + \frac C y  \iint_{ B_R  \times (0,\tau) }|b| m \dxdt .
\end{equation}
In particular, by \eqref{massloss}, setting $k_1 = C\sigma \tau$ and $k_2 = C  \iint_{ B_R  \times (0,\tau) }|b| m$,
\begin{equation}\label{eqzz2}
\int_{|x| \ge y} m(x, \tau) \dx \le \min \left\{ \frac {k_1}{y^2} + \frac {k_2} y , 1 \right\}.
\end{equation}
Now,
\begin{multline*}
\int_{B_R} |x|^\alpha m(x,\tau) \dx = \alpha \int_{B_R} \int_0^{|x|} y^{\alpha-1} m(x,\tau) \dy \dx = \alpha  \int_0^R y^{\alpha-1} \int_{|x| \ge y}  m(x,\tau) \dx \dy \\
\stackrel{\eqref{eqzz2}}{\le} \alpha \int_0^R y^{\alpha-1} \min \left\{  {k_1} {y^{-2}} , 1 \right\} \dy + \alpha \int_0^R y^{\alpha-1} \min \left\{ k_2 {y^{-1}} , 1 \right\} \dy \\
\le k_1 \alpha  \int_{k_1^{1/2}}^{+\infty} y^{\alpha-3}  \dy + \alpha  \int_0^{k_1^{1/2}} y^{\alpha-1}  \dy + k_2 \alpha  \int_{k_2}^{+\infty} y^{\alpha-2}  \dy + \alpha  \int_0^{k_2} y^{\alpha-1}  \dy \\
= \alpha \frac{k_1^{\alpha/2}}{2-\alpha} + k_1^{\alpha/2} + \alpha \frac{k_2^{\alpha}}{1-\alpha} + k_2^{\alpha}.
\end{multline*}
Substituting $k_1,k_2$ gives the assertion.
\end{proof} 

As a consequence \eqref{massloss} and \eqref{eqzz} (with $y = R$), we have
\begin{equation}\label{boundaryloss}
- \sigma \iint_{\partial B_R  \times (0,\tau) } Dm \cdot \nu \dxdt \le C \frac{\tau^{1/\gamma}}R \left( \iint_{B_R  \times (0,\tau) } |b|^{\gamma'} m \dxdt \right )^{1/\gamma'}+ C \frac{\sigma \tau}{R^2}
\end{equation}
and by the previous lemma,
\begin{equation}\label{bdglemma2}
\int_{B_R} |x|^\alpha m(x,\tau) \dx \le  C \tau^{\alpha/\gamma}   \left( \iint_{B_R  \times (0,\tau) } |b|^{\gamma'} m \dxdt  \right)^{\alpha/\gamma'}+ C \left(\sigma \tau \right)^{\alpha/2}
\end{equation}

\begin{lemma}\label{lemmamq0} Suppose that $R^2 \ge \tau \sigma$. There exists $C$ depending on $q_0$ such that
\begin{equation}\label{mLq0}
\sigma^{\gamma'\frac{N+1}{N+2}}\, \|m\|_{L^{q'_0}(Q_{R,\tau})} \le C\left( \iint_{ B_R  \times (0,\tau) } |b|^{\gamma'} m \dxdt + \sigma^{\gamma'/2}\tau^{\alpha_0/2} \right).
\end{equation}
\end{lemma}

\begin{proof} We start with the case $\sigma = R = 1$, following the same lines as \cite[Proposition 2.4]{CT}, that is
\begin{equation}\label{FPz}
\begin{cases}
\partial_t \tilde m - \Delta \tilde m - \divv(\tilde b \tilde m) = 0& \text{on $B_1 \times (0,\bar s)$} \\
m(z,s) = 0 & \text{on $\partial B_1 \times (0,\bar s)$} \\
m(0) = \delta_0,
\end{cases}
\end{equation}
for any $\bar s \in (0,1]$. We might assume that $\tilde m$ is smooth, the general case follows by approximation. Let $\zeta$ be the solution of the dual problem
\[
\begin{cases}
-\partial_t \zeta - \Delta \zeta = \tilde m^{q_0'-1}& \text{on $B_1 \times (0,\bar s)$} \\
\zeta(z,s) = 0 & \text{on $\partial B_1 \times (0,\bar s)$} \\
\zeta(z,\bar s) = 0 & \text{on $B_1$}.
\end{cases}
\]
We now employ $L^{q_0}$-parabolic regularity, and embedding theorems of $W^{2,1}_{q_0}(Q_{1,\bar s})$ into $W^{1,0}_{q}(Q_{1,\bar s})$ and $C^{\alpha_0/2}(0,\bar s; C^{\alpha_0}(B_1))$, where $q^{-1} = q_0^{-1} - (N+2)^{-1}$ and $\alpha_0 = 2 - \frac{N+2}{q_0}= 2-\gamma'$ (see \cite[Theorem IV.9.1]{LSU} and \cite[Lemma II.3.3]{LSU}, recall that $\frac{N+2}2<q_0<N+2$), to get
\begin{align}\label{embedd}
|\zeta(0,0)|  \le C'\bar s^{\alpha_0/2} \|\zeta\|_{W^{2,1}_{q_0}(Q_{1,\bar s})} \le & C\bar s^{\alpha_0/2} \|\tilde m^{q'_0-1} \|_{L^{q_0}(Q_{1,\bar s})} = C\bar s^{\alpha_0/2} \|\tilde m\|^{q'_0-1}_{L^{q'_0}(Q_{1,\bar s})} , \\
\|D\zeta\|_{L^{q}(Q_{1,\bar s})} \le C' \|\zeta\|_{W^{2,1}_{q_0}(Q_{1,\bar s})} \le & C \|\tilde m^{q'_0-1}\|_{L^{q_0}(Q_{1,\bar s})} = C \|\tilde m\|^{q'_0-1}_{L^{q'_0}(Q_{1,\bar s})} .
\end{align}
Note that $C,C'$ do not depend on $\bar s \le 1$. Then, by duality between $\tilde m$ and $\zeta$,
\[
\iint_{Q_{1,\bar s}} \tilde m^{q'_0} = |\zeta(0,0)| - \iint_{Q_{1,\bar s }} \tilde b \tilde m \cdot D\zeta.
\]
Hence, by H\"older inequality ($1=\frac 1 {\gamma'} + \frac 1 {\gamma q_0'} + \frac 1 q $),
\begin{align*}
\|\tilde m\|^{q'_0}_{L^{q'_0}(Q_{1,\bar s})} & \le C\|\tilde b \tilde m^{1/\gamma'}\|_{L^{\gamma'}(Q_{1,\bar s})} \|\tilde m^{1/\gamma}\|_{L^{\gamma q'_0}(Q_{1,\bar s})} \|D\zeta\|_{L^{q}(Q_{1,\bar s})} +  |\zeta(0,0)| \\ 
& \stackrel{\eqref{embedd}}{\le} C\left(\left( \iint_{Q_{1,\bar s}} |\tilde b|^{\gamma'} \tilde m \right)^{1/\gamma'} \|\tilde m\|^{1/\gamma}_{L^{q'_0}(Q_{1,\bar s})} + \bar s^{\alpha_0/2} \right)\|\tilde m\|^{q'_0-1}_{L^{q'_0}(Q_{1,\bar s})},
\end{align*}
and the desired inequality  \eqref{mLq0} follows by means of Young's inequality.

\smallskip

The general case, for $m$ solving \eqref{FP}, can be obtained by scaling: setting $\tilde m(z,s) = R^{N} m(Rz, R^2\sigma^{-1}s)$, $\tilde b(z,s) = R \sigma^{-1} b(Rz, R^2\sigma^{-1}s)$, and $\bar s = \tau R^{-2}\sigma \le 1$, then \eqref{FPz} holds. Therefore, by a straightforward computation
\begin{multline*}
\sigma^{1/{q_0'}} R^{N(1-1/q_0')-2/q_0'}\|m\|_{L^{q'_0}(Q_{R, \tau})} = \|\tilde m\|_{L^{q'_0}(Q_{1,\bar s})} \le C\left( \iint_{Q_{1,\bar s}} |\tilde b|^{\gamma'} \tilde m \dif z \dif s  + \bar s^{\alpha_0/2} \right) = \\
= C\left(\frac{R^{\gamma'-2}}{\sigma^{\gamma'-1}} \iint_{Q_{R, \tau}} |b|^{\gamma'}  m \dif x \dif t  + \frac{\sigma^{\alpha_0/2}}{R^{\alpha_0}}\tau^{\alpha_0/2} \right),
\end{multline*}
which yields the conclusion, since $N(1-1/q_0')-2/q_0'+2-\gamma' = 0$, $\alpha_0 = 2-\gamma'$ and $\alpha_0/2 + \gamma' - 1 = \gamma'/2$.
\end{proof}

We now state a property of $\zeta \mapsto |\zeta|^{\gamma'}$.

\begin{lemma} For any $\gamma' \in (1,2)$, there exists $C$ depending on $\gamma'$ such that
\begin{equation}\label{Ldiff}
|\zeta+\xi|^{\gamma'}-|\zeta|^{\gamma'} \le C(|\zeta|^{\gamma'-1}|\xi| + |\xi|^{\gamma'})
\end{equation}
for all $\zeta, \xi \in \R^N$.
\end{lemma}

\begin{proof} Since $\R^n \ni p \mapsto |p|^{\gamma'}$ is convex,
\[
|\zeta+\xi|^{\gamma'}-|\zeta|^{\gamma'} \le \gamma' |\zeta+\xi|^{\gamma'-2} (\zeta+\xi) \cdot \xi.
\]
If $|\zeta| \le |\xi|$,
\[
|\zeta+\xi|^{\gamma'}-|\zeta|^{\gamma'} \le \gamma' |\zeta+\xi|^{\gamma'-1}|\xi| \le \gamma'2^{\gamma'-1}|\xi|^{\gamma'}.
\]
Else, $|\zeta| > |\xi|$; since $(0, \infty) \ni t \mapsto t^{\gamma'-1}$ is concave, and $(0, \infty) \ni t \mapsto t^{\gamma'-2}$ is decreasing,
\begin{align*}
|\zeta+\xi|^{\gamma'}-|\zeta|^{\gamma'} & \le \gamma' (|\zeta+\xi|^{\gamma'-1}- |\zeta|^{\gamma'-1})|\xi| + \gamma'  |\zeta|^{\gamma'-1}|\xi| \\
& \le \gamma'(\gamma'-1) |\zeta|^{\gamma'-2}(|\zeta+\xi|-|\zeta|)|\xi| + \gamma'  |\zeta|^{\gamma'-1}|\xi| \\
& \le \gamma'(\gamma'-1) |\zeta|^{\gamma'-2}|\xi|^2 + \gamma'  |\zeta|^{\gamma'-1}|\xi|\\
& \le \gamma'(\gamma'-1) |\xi|^{\gamma'} + \gamma'  |\zeta|^{\gamma'-1}|\xi|.
\end{align*}
\end{proof}

We finally present a crucial result on the $W^{2,1}_q$ regularity of solutions to HJ equations enjoying suitable H\"older bounds. This is a parabolic analogue of \cite[Proposition 3.3]{CV}.

\begin{proposition}\label{prop:holdertoW2q} Let $R > 0 $, $g \in L^q(Q_{2R})$ and $v \in W^{2,1}_q(Q_{2R})$ be such that $v(0) = 0$,
\[
\|g\|_{L^q(Q_{2R})} + [v]_{\alpha; Q_{2R}} \le c_1\qquad \left|- \partial_t v - \Delta v \right| \le c_2|D v|^\gamma + g
\]
a.e. in $Q_{2R}$, for some $c_1, c_2 > 0$, and  $\alpha+\frac{N+2}{q} = 2$. Then, there exists $K$ depending on $c_1, c_2, R, N, q$ such that
\[
\|\partial_t v\|_{L^q( Q_{R})} + \|D^2 v\|_{L^q(Q_{R})}\le K.
\]
Moreover, the same statement holds if $Q_R$ and $Q_{2R}$ are replaced by cylinders $Q' \ssubset Q$ respectively.
\end{proposition}
\begin{proof} First, since $[v]_{\alpha; Q_{2R}} \le c_1$ and $v(0) = 0$,
\begin{equation}\label{eq:boundwn1}
\|v\|_{L^q(Q_{2R})}\le c_1 C  R^{\alpha+\frac{N+2}{q}} = c_1 C R^2.
\end{equation}
Let $R \le \rho \le 2R$ and $0 < \sigma < 1$. Arguing as in the proof of \cite[Theorem 7.22]{Lieberman}, and using H\"older's inequality and \eqref{eq:boundwn1}
\begin{equation}\label{eq:key_interm}
\begin{split}
\|D^2 v\|_{L^q(Q_{\sigma \rho})}&\le \frac{C}{(1-\sigma)^2R^2}\left(
R^2 \|  c_2|D v|^\gamma + g \|_{L^q(Q_{\rho})} + \|  v \|_{L^q(Q_{\rho})}
\right)\\
&\le \frac{C}{(1-\sigma)^2}\left(
R^{\gamma-\frac{N+2}{q}(\gamma-1)}\|D v\|_{p;Q_\rho}^\gamma+1
\right).
\end{split}
\end{equation}
where $\frac1p = \frac1q - \frac1{N+2}$. By the Gagliardo-Nirenberg interpolation inequality \cite{Nirenberg_1966}, there exists
$C$ (independent of $ \rho$) such that for a.e. $t \in (0, \rho^2)$,
\[
\|Dv(t)\|_{p ;B_\rho} \le C\left(\|D^2 v(t)\|_{L^q(B_\rho)}^a[v(t)]_{\alpha;B_\rho}^{1-a} + [v(t)]_{\alpha;B_\rho} \right),
\qquad \text{where }a=1-\frac{q}{N+2}<\frac{1}{\gamma} .
\]
Since $ap = q$, if we raise the previous inequality to the power $p$ and integrate in time we obtain
\[
\|D v\|^p_{p;Q_\rho}\le C (\|D^2 v\|_{L^q(Q_\rho)}^q  +  R^2).
\]
If $\|D^2 v\|_{L^q(Q_R)}^q  \le  R^2$, then there is nothing else to prove, since $\|D^2 v\|_{q,Q_R}$ is bounded, and a control on $\|\partial_t v\|_{q,Q_R}$ follows immediately by integrating the differential inequality for $v$.
Otherwise, plugging the resulting inequality into \eqref{eq:key_interm} implies
\begin{equation}\label{eq:smallvsbig}
\|D^2 v\|_{L^q(Q_{\sigma \rho})} \le  \frac{C}{(1-\sigma)^2}\left(R^{\gamma-\frac{N+2}{q}(\gamma-1)} 
\|D^2 v\|^{a\gamma}_{L^q(Q_\rho)}+1
\right).
\end{equation}
\smallskip

As before, if $R^{\gamma-\frac{N+2}{q}(\gamma-1)} 
\|D^2 v\|^{a\gamma}_{L^q(Q_R)} \le 1$, then we are done. Otherwise, \eqref{eq:smallvsbig} yields, for every $R\le \rho \le
2R$ and $0<\sigma<1$,
\begin{equation}\label{eq:iter_1}
\|D^2 v\|_{L^q(Q_{\sigma \rho})}\le\frac{ E}{(1-\sigma)^2}
\|D^2 v\|_{L^q(Q_\rho)}^{a\gamma} \ ,
\end{equation}
where $E =CR^{\gamma-\frac{N+2}{q}(\gamma-1)}$, and $C$ is independent of $\rho,\sigma$. We may then conclude following exactly the same lines as in the proof of \cite[Proposition 3.3]{CV}; the statement with $Q', Q$ follows by a standard covering argument.
\end{proof}

\small

\begin{thebibliography}{10}

\bibitem{CarCan}
P.~Cannarsa and P.~Cardaliaguet.
\newblock H\"{o}lder estimates in space-time for viscosity solutions of
  {H}amilton-{J}acobi equations.
\newblock {\em Comm. Pure Appl. Math.}, 63(5):590--629, 2010.

\bibitem{CPT}
P.~Cardaliaguet, A.~Porretta, and D.~Tonon.
\newblock Sobolev regularity for the first order {H}amilton-{J}acobi equation.
\newblock {\em Calc. Var. Partial Differential Equations}, 54(3):3037--3065,
  2015.

\bibitem{CSil}
P.~Cardaliaguet and L.~Silvestre.
\newblock H\"{o}lder continuity to {H}amilton-{J}acobi equations with
  superquadratic growth in the gradient and unbounded right-hand side.
\newblock {\em Comm. Partial Differential Equations}, 37(9):1668--1688, 2012.

\bibitem{CG2}
M.~Cirant and A.~Goffi.
\newblock Lipschitz regularity for viscous {H}amilton-{J}acobi equations with
  {$L^p$} terms.
\newblock {\em Ann. Inst. H. Poincar\'{e} Anal. Non Lin\'{e}aire},
  37(4):757--784, 2020.

\bibitem{CGpar}
M.~Cirant and A.~Goffi.
\newblock Maximal {$L^q$}-regularity for parabolic {H}amilton-{J}acobi
  equations and applications to mean field games.
\newblock {\em Ann. PDE}, 7(2):Paper No. 19, 40, 2021.

\bibitem{CG4}
M.~Cirant and A.~Goffi.
\newblock On the problem of maximal ${L}^q$-regularity for viscous
  {H}amilton-{J}acobi equations.
\newblock {\em Arch. Rational Mech. Anal.}, 2021.
\newblock doi:10.1007/s00205-021-01641-8.

\bibitem{CT}
M.~Cirant and D.~Tonon.
\newblock Time-dependent focusing mean-field games: the sub-critical case.
\newblock {\em J. Dynam. Differential Equations}, 31(1):49--79, 2019.

\bibitem{CV}
M.~Cirant and G.~Verzini.
\newblock Local {H}\"older and maximal regularity of solutions of elliptic
  equations with superquadratic gradient terms.
\newblock arXiv:2203.06092, 2022.

\bibitem{DAP}
A.~Dall'Aglio and A.~Porretta.
\newblock Local and global regularity of weak solutions of elliptic equations
  with superquadratic {H}amiltonian.
\newblock {\em Trans. Amer. Math. Soc.}, 367(5):3017--3039, 2015.

\bibitem{Evans}
L.~C. Evans.
\newblock Adjoint and compensated compactness methods for {H}amilton-{J}acobi
  {PDE}.
\newblock {\em Arch. Ration. Mech. Anal.}, 197(3):1053--1088, 2010.

\bibitem{G}
A.~Goffi.
\newblock On the optimal ${L}^q$-regularity for viscous {H}amilton-{J}acobi
  equations with sub-quadratic growth in the gradient.
\newblock arXiv:2112.02676, 2021.

\bibitem{GoffiFrac}
A.~Goffi.
\newblock Transport equations with nonlocal diffusion and applications to
  {H}amilton-{J}acobi equations.
\newblock {\em J. Evol. Equ.}, 21(4):4261--4317, 2021.

\bibitem{GP}
A.~Goffi and F.~Pediconi.
\newblock Sobolev regularity for nonlinear {P}oisson equations with {N}eumann
  boundary conditions on {R}iemannian manifolds.
\newblock arXiv:2110.15450, 2021.

\bibitem{Gomesbook}
D.~A. Gomes, E.~A. Pimentel, and V.~Voskanyan.
\newblock {\em Regularity theory for mean-field game systems}.
\newblock SpringerBriefs in Mathematics. Springer, [Cham], 2016.

\bibitem{LSU}
O.~A. Ladyzenskaja, V.~A. Solonnikov, and N.~N. Ural'tseva.
\newblock {\em Linear and quasilinear equations of parabolic type}.
\newblock Translated from the Russian by S. Smith. Translations of Mathematical
  Monographs, Vol. 23. American Mathematical Society, Providence, R.I., 1968.

\bibitem{Lieberman}
G.~M. Lieberman.
\newblock {\em Second order parabolic differential equations}.
\newblock World Scientific Publishing Co. Inc., River Edge, NJ, 1996.

\bibitem{Lions85}
P.-L. Lions.
\newblock Quelques remarques sur les probl\`emes elliptiques quasilin\'{e}aires
  du second ordre.
\newblock {\em J. Analyse Math.}, 45:234--254, 1985.

\bibitem{Nirenberg_1966}
L.~Nirenberg.
\newblock An extended interpolation inequality.
\newblock {\em Ann. Scuola Norm. Sup. Pisa Cl. Sci. (3)}, 20:733--737, 1966.

\bibitem{SP}
L.~A. Peletier and J.~Serrin.
\newblock Gradient bounds and {L}iouville theorems for quasilinear elliptic
  equations.
\newblock {\em Ann. Scuola Norm. Sup. Pisa Cl. Sci. (4)}, 5(1):65--104, 1978.

\bibitem{PQS}
P.~Pol\'{a}\v{c}ik, P.~Quittner, and P.~Souplet.
\newblock Singularity and decay estimates in superlinear problems via
  {L}iouville-type theorems. {II}. {P}arabolic equations.
\newblock {\em Indiana Univ. Math. J.}, 56(2):879--908, 2007.

\bibitem{PSou}
A.~Porretta and P.~Souplet.
\newblock Blow-up and regularization rates, loss and recovery of boundary
  conditions for the superquadratic viscous {H}amilton-{J}acobi equation.
\newblock {\em J. Math. Pures Appl. (9)}, 133:66--117, 2020.

\bibitem{SZ}
P.~Souplet and Q.~S. Zhang.
\newblock Global solutions of inhomogeneous {H}amilton-{J}acobi equations.
\newblock {\em J. Anal. Math.}, 99:355--396, 2006.

\bibitem{SZ06}
P.~Souplet and Q.~S. Zhang.
\newblock Sharp gradient estimate and {Y}au's {L}iouville theorem for the heat
  equation on noncompact manifolds.
\newblock {\em Bull. London Math. Soc.}, 38(6):1045--1053, 2006.

\bibitem{SV18}
L.~F. Stokols and A.~F. Vasseur.
\newblock De {G}iorgi techniques applied to {H}amilton-{J}acobi equations with
  unbounded right-hand side.
\newblock {\em Commun. Math. Sci.}, 16(6):1465--1487, 2018.

\end{thebibliography}

\medskip
\small
\begin{flushright}
\noindent Marco Cirant\\
Dipartimento di Matematica ``Tullio Levi-Civita'', Universit\`a di Padova\\
Via Trieste 63, 35121 Padova, Italy\\
\verb"cirant@math.unipd.it"
\end{flushright}

\end{document}